\pdfoutput=1
\documentclass{article}
\usepackage{amsmath, amsfonts, amssymb, amsthm,
  setspace, yfonts, bm,
  mathtools, enumitem}
\usepackage{latexsym}
\usepackage{authblk}

\newcommand{\pg}{\mathop {\rm PG}}

\newcommand{\ag}{\mathop {\rm AG}}

\newcommand{\st}{\mathop {\rm star}}
\newcommand{\lne}{\mathop {\rm line}}
\newcommand{\pen}{\mathop {\rm pencil}}

\newcommand{\rmB}{\mathop {\rm B}}
\newcommand{\rmN}{\mathop {\rm N}}
\newcommand{\rmQ}{\mathop {\rm Q}}
\newcommand{\rmT}{\mathop {\rm T}}

\newcommand{\rmi}{\mathrm{i}}

\newcommand{\bbC}{\mathbb{C}}
\newcommand{\bbE}{\mathbb{E}}
\newcommand{\bbF}{\mathbb{F}}
\newcommand{\bbT}{\mathbb{T}}
\newcommand{\bbZ}{\mathbb{Z}}

\newcommand{\cK}{\mathcal{K}}
\newcommand{\cL}{\mathcal{L}}
\newcommand{\cO}{\mathcal{O}}
\newcommand{\cP}{\mathcal{P}}
\newcommand{\cQ}{\mathcal{Q}}
\newcommand{\cS}{\mathcal{S}}
\newcommand{\cT}{\mathcal{T}}
\newcommand{\inv}{^{-1}}

\theoremstyle{plain} \numberwithin{equation}{section}
\newtheorem{theorem}{Theorem}[section]
\newtheorem{corollary}[theorem]{Corollary}
\newtheorem{conjecture}{Conjecture}
\newtheorem{lemma}[theorem]{Lemma}

\theoremstyle{definition}
\newtheorem{definition}[theorem]{Definition}

\newtheorem{remark}[theorem]{Remark}

\newcommand\blfootnote[1]{%
  \begingroup
  \renewcommand\thefootnote{}\footnote{#1}%
  \addtocounter{footnote}{-1}%
  \endgroup
}

\begin{document}

\title{A new family of tight sets in $\cQ^{+}(5,q)$ \blfootnote{
J.~De Beule is a postdoctoral fellow of the Research Foundation Flanders -- Belgium (FWO). J.~Demeyer has been supported as a postdoctoral fellow of the Research Foundation Flanders -- Belgium (FWO). The research of M.~Rodgers has been supported partially by the FWO project "Moufang verzamelingen" G.0140.09
}}
\author[1]{Jan De Beule}
\author[1]{Jeroen Demeyer}
\author[1,2]{Klaus Metsch}
\author[1]{Morgan Rodgers }
\affil[1]{Ghent University\\
  Department of Mathematics\\
  Krijgslaan 281, S22\\
  B--9000 Gent\\
  Belgium\\
  jdebeule@cage.ugent.be\\
  jdmeyer@cage.ugent.be\\
  morgan.joaquin@gmail.com }
\affil[2]{ Universit\"at Gie\ss{}en \\
  Mathematisches Institut\\
  Arndtstra\ss{}e 2\\
  D--35392 Gie\ss{}en\\
  Germany\\
  klaus.metsch@math.uni-giessen.de}
\renewcommand\Authands{ and }

\maketitle

\begin{abstract}
In this paper, we describe a new infinite family of
$\frac{q^{2}-1}{2}$-tight sets in the hyperbolic quadrics
$\cQ^{+}(5,q)$, for $q \equiv 5 \mbox{ or } 9 \bmod{12}$.  Under the
Klein correspondence, these correspond to Cameron--Liebler line
classes of $\pg(3,q)$ having parameter $\frac{q^{2}-1}{2}$.  This is
the second known infinite family of nontrivial Cameron--Liebler line
classes, the first family having been described by Bruen and Drudge
with parameter $\frac{q^{2}+1}{2}$ in $\pg(3,q)$ for all odd $q$.

The study of Cameron--Liebler line classes is closely related to the
study of symmetric tactical decompositions of $\pg(3,q)$ (those having
the same number of point classes as line classes).  We show that our
new examples occur as line classes in such a tactical decomposition
when $q \equiv 9 \bmod 12$ (so $q = 3^{2e}$ for some positive integer
$e$), providing an infinite family of counterexamples to a conjecture
made by Cameron and Liebler in 1982; the nature of these
decompositions allows us to also prove the existence of a set of type
$\left(\frac{1}{2}(3^{2e}-3^{e}), \frac{1}{2}(3^{2e}+3^{e}) \right)$ in the affine
plane $\ag(2,3^{2e})$ for all positive integers $e$.  This proves a
conjecture made by Rodgers in his PhD thesis.


\end{abstract}

\section{Introduction}
Let $q=p^h$, $p$ prime, $h \geq 1$, and let $\pg(d,q)$ denote the
$d$-dimensional projective space over the finite field $\bbF_q$. A
{\em spread} of $\pg(3,q)$ is a set $\cS$ of lines of $\pg(3,q)$ such
that every point of $\pg(3,q)$ is contained in exactly one line of
$\cS$, i.e. the set of lines of $\cS$ partitions the set of points of
$\pg(3,q)$. Spreads of $\pg(3,q)$ are well studied, and many families
of examples exist. 

Cameron and Liebler in \cite{CL1982} originally studied irreducible
collineation groups of $\pg(d,q)$ that have equally many point orbits
as line orbits.  This problem ended up being closely related to the
study of symmetric tactical decompositions (tactical decompositions
having the same number of point classes as line classes) of $\pg(d,q)$
(viewed as a point-line design), and this work was further generalized
to the study of line sets later termed Cameron--Liebler line classes
(the following definition is specific for $\pg(3,q)$, though these
objects are also defined in higher dimensional projective spaces).
\begin{definition}
Let $\bm{A}$ be the point-line adjacency matrix of $\pg(d,q)$.  A set
of lines $\cL$ is called a \emph{Cameron--Liebler line class} if the
characteristic vector $\bm{c}_{\cL}$ of $\cL$ lies in $\mathrm{row}(A)$.
\end{definition}
The connection between these three concepts is as follows; a
collineation group of $\pg(d,q)$ having equally many orbits on points
and lines induces a symmetric tactical decomposition on $\pg(d,q)$,
and any line class of such a tactical decomposition is a
Cameron--Liebler line class.  

Cameron--Liebler line classes of $\pg(3,q)$ have received considerable
attention, and this is the situation we will focus on in this work.
One important characterization here is that $\cL$ is a 
Cameron--Liebler line class of $\pg(3,q)$ if and only if any spread of
$\pg(3,q)$ contains exactly $x$ lines of $\cL$ for some integer $x$
called the \emph{parameter} of $\cL$.  A spread of $\pg(3,q)$ contains
$q^{2}+1$ lines, so $0 \leq x \leq q^2+1$.   Since any spread
partitions the set of points of $\pg(3,q)$, it is clear immediately
that the set of lines through a given point is a Cameron--Liebler line
class with parameter $1$. The same holds for its dual: the set of
lines in a given plane is a Cameron--Liebler line class with parameter
$1$. Note that the union of two disjoint Cameron--Liebler line
classes with respective parameters $x$ and $y$ is a Cameron--Liebler
line class with parameter $x+y$, and that the complement of a
Cameron--Liebler line class with parameter $x$ is a Cameron--Liebler
line class with parameter $q^2+1-x$. Hence the union of the set of
lines through a point and the set of lines in a plane not containing
that point is a Cameron--Liebler line class with parameter $2$. The
examples seen so far (including their complements) are called trivial
examples.  

In their work, Cameron and Liebler put forward the following
progressively weaker conjectures: 
\begin{conjecture}\label{conj:CL}
The only Cameron--Liebler line classes in $\pg(d,q)$ are the trivial
examples.
\end{conjecture}
\begin{conjecture}\label{conj:STD}
A symmetric tactical decomposition of $\pg(d,q)$ consists of either 
\begin{enumerate}[label=\rm{(\roman*)}, ref=\rm{(\roman*)}]
\item a single point and line class;
\item two point classes $\{ \bm{p}\}$, $\pg(d,q) \setminus \{ \bm{p}
  \}$ and two line classes $\st(\bm{p})$, $\st(\bm{p})^{C}$ for some
  point $\bm{p}$; or (dually)
\item two point classes $H$, $\pg(d,q) \setminus H$ and two line
  classes $\lne(H)$, $\lne(H)^{C}$ for some hyperplane $H$.
\end{enumerate}
\end{conjecture}
\begin{conjecture}\label{conj:ColGrp}
An irreducible collineation group of $\pg(d,q)$ having equally many
orbits on points and lines either 
\begin{enumerate}[label=\rm{(\roman*)}, ref=\rm{(\roman*)}]
\item is line-transitive;
\item stabilizes a hyperplane $\pi$ and acts line-transitively on it;
  or (dually) 
\item fixes a point $\bm{p}$ and acts line-transitively on the
  quotient space.
\end{enumerate}
\end{conjecture}

Conjecture~\ref{conj:CL} was disproved for the first time in
\cite{D1999} with the construction of a Cameron--Liebler line class
with parameter $5$ in $\pg(3,3)$.  Later in \cite{BD1999}, the authors
construct an infinite family of Cameron--Liebler line classes. More
particularly, they show the existence of a Cameron--Liebler line class
with parameter $\frac{q^{2}+1}{2}$ for odd $q$. Essentially, their
construction is done in a geometric way; the Cameron--Liebler
line-class is obtained as the union of all secant lines to a fixed
elliptic quadric in $\pg(3,q)$, and a well defined and cleverly chosen
subset of the tangent lines to the same elliptic quadric.  In
\cite{GPCL05}, another example was constructed in $\pg(3,4)$
having parameter $7$.  None of these examples arise as line classes in
a symmetric tactical decomposition, however.

Meanwhile, the years after \cite{CL1982} have also seen the appearance
of many non-existence results (\cite{Pentt1991}, \cite{D1999},
\cite{GPCL05}, \cite{GS2004}, \cite{BHS2008}, \cite{BGHS2009},
\cite{M2010}, \cite{BM2013}).  Currently, the most general result
(which includes all the previous ones) is found in
\cite{Metsch2014}. Its main result is the following: 
if $\cL$ is a Cameron--Liebler line class with parameter $x$ then $x
\leq 2$ or $x > q\sqrt[4]{\frac{q}{2}}-\frac{2}{3}q$. This result
improves all previous non-existence results except for a few cases
dealing with small values of $q$. 

In this paper, we prove the existence of Cameron--Liebler line classes
with parameter $\frac{q^2-1}{2}$ for $q \equiv 5,9 \bmod 12$,  
which is, since \cite{BD1999}, the first construction of an infinite
family. These examples were first found by Morgan Rodgers,
\cite{Rodgers}. His examples were obtained under the assumption of
symmetry conditions for a hypothetical line class, then proceeding
through eigenvalue methods and finally obtaining the examples through
a computer search.  It was observed that the examples constructed
in $\pg(3,9)$ and $\pg(3,81)$ did in fact arise as line classes of a
symmetric tactical decomposition having four classes on points and
lines, providing the first counterexamples to
Conjecture~\ref{conj:STD}.   Here, we show that this is the case for
all of our examples with $q \equiv 9 \bmod{12}$.
\begin{remark}
It is worth noting that Conjecture~\ref{conj:ColGrp} was proven
recently in \cite{BamPentt}.
\end{remark}

\section{Tight sets, eigenvectors, and tactical decompositions}

Consider a $6$-dimensional vector space $V(6,q)$ over the finite field
$\bbF_q$. Let $f: V(6,q) \rightarrow \bbF_q$ be a bilinear form of
Witt index $3$. Hence the maximal totally isotropic subspaces with
relation to $f$ have dimension $3$. The totally isotropic subspaces
with relation to $f$ induces in $\pg(5,q)$ a set of points, lines and
planes contained in a quadratic surface. Up to coordinate
transformation, there is only one bilinear form of Witt index $3$, and
so up to coordinate transformation, there is only one such quadratic
surface. We call this geometry the hyperbolic quadric in $\pg(5,q)$,
and denote it as $\cQ^+(5,q)$. The projective spaces of maximal
dimension contained in $\cQ^+(5,q)$ are also called {\em
  generators}. Hence, the generators of $\cQ^+(5,q)$ are planes. The
generators of $\cQ^+(5,q)$ come in two systems. Two distinct
generators from the same system meet in a point, two generators from
different systems meet either in a line or are skew. Hence two skew
generators of $\cQ^+(5,q)$ necessarily belong to  a different system,
and no three generators mutually skew can be found on $\cQ^+(5,q)$. We
refer to e.g.\ \cite{HirschI} to find all these well known properties
of quadrics in finite projective spaces.  Two points $\bm{p}, \bm{q}$
of $\cQ^{+}(5,q)$ are collinear if $f(u,v) = 0$ with $u,v$ being
vector representatives of $\bm{p},\bm{q}$. We denote $\bm{p}^{\perp}$
the set of points of $\cQ^+(5,q)$ collinear with the point $\bm{p}$
(this includes $\bm{p}$ itself).
\begin{definition} A set $\cT$ of points of $\cQ^{+}(5,q)$ is an
  \emph{$x$-tight set} if for every point $\bm{p} \in \cQ^{+}(5,q)$ we
  have
\[
|\bm{p}^{\perp} \cap \cT| = x(q+1) + q^{2}\bm{c}_{\cT}(\bm{p}).
\]
\end{definition}
The {\em Klein correspondence} is a bijection from the set of lines of
$\pg(3,q)$ to the set of points of $\cQ^+(5,q)$. The image of a line
is the point with coordinates the Pl\"ucker coordinates of the line,
and concurrent lines are mapped to collinear points. Generators 
correspond with the image of either the set of lines through a point,
or the set of lines in a plane. The image of a Cameron--Liebler line
class with parameter $x$ is an $x$-tight set of $\cQ^+(5,q)$. Hence,
the trivial examples are (the union of two skew) generators and their
complements. Note that tight sets were introduced in
\cite{Payne1987}.

In constructing $i$-tight sets of $\cQ^{+}(5,q)$, it is helpful to use
the following result, due to Bamberg, Kelly, Law and Penttila 
\cite{BKLP07}.
\begin{theorem}\label{thm:Evec}
Let $\cT$ be a set of points in $\cQ^{+}(5,q)$ with characteristic
vector $\bm{c}$ and let $\bm{A}$ be the collinearity matrix of
$\cQ^{+}(5,q)$. Then $\cT$ is an $x$-tight set if and only if  
\[ (\bm{c} - \frac{x}{q^{2}+1} \bm{j}) \bm{A} = (q^{2} -1)(\bm{c}
-\frac{x}{q^{2}+1} \bm{j}) \mbox{.} \]
\end{theorem}
\begin{proof}
By definition, $\cT$ is an $x$-tight set if and only if
\[ 
\bm{c} \bm{A} = (q^{2} -1) \bm{c} + x(q+1) \bm{j}.
\]
Since $\bm{j} \bm{A} = q (q+1)^{2} \bm{j}$, the above formula follows 
immediately. 
\end{proof}
Our main theorem concerns tight sets which are disjoint from $(\pi_{1} 
\cup \pi_{2})$, where $\pi_{1}$ and $\pi_{2}$ are two generators which are disjoint.
This allows us to modify this theorem slightly.  We will let
$\bm{A}^{\prime}$ be the matrix obtained from $\bm{A}$ by throwing
away the rows and columns corresponding to points in  $(\pi_{1} \cup
\pi_{2})$. 
\begin{theorem}
\label{EvecMod}
Let $\cT$ be a set of points of $\cQ^{+}(5,q)$ disjoint from $(\pi_{1}
\cup \pi_{2})$, and let $\bm{c}^{\prime}$ be the vector obtained
from the characteristic vector of $\cT$ by removing entries
corresponding to points of $(\pi_{1} \cup \pi_{2})$.  Then $\cT$ is an
$x$-tight set of $\cQ^{+}(5,q)$ if and only if
\[ 
(\bm{c}^{\prime} - \frac{x}{q^{2} -1} \bm{j}^{\prime}) \bm{A}^{\prime}
= (q^{2}-1)(\bm{c}^{\prime} - \frac{x}{q^{2} -1} \bm{j}^{\prime}). 
\]
\end{theorem}
\begin{proof}
For a vector $\bm{v}$, denote by $\bm{v}^{\prime}$ the vector obtained
from $\bm{v}$ by deleting entries corresponding to points in $\pi_{1}
\cup \pi_{2}$.  Then $(\bm{c}\bm{A})^{\prime} =
\bm{c}^{\prime}\bm{A}^{\prime}$ since no point of $\cT$ lies in
$\pi_{1} \cup \pi_{2}$, and $(\bm{j}\bm{A})^{\prime} = \bm{j}^{\prime}
\bm{A}^{\prime} + 2(q+1)\bm{j}^{\prime}$. The result then follows from
Theorem~\ref{thm:Evec}. 
\end{proof}
Of course $\cQ^{+}(5,q)$ is very large, having $(q^{2}+1)(q^{2}+q+1)$
points.  The eigenspace of the collinearity matrix associated with the
eigenvector $(q^{2}-1)$ has dimension $q(q^{2}+q+1)$, and finding
vectors in this space corresponding to tight sets amounts to solving a 
difficult integer programming problem.  Thus our work is made easier
by considering various decompositions of the incidence structure, as
well as decompositions of the incidence matrix.  

\begin{definition}
  Given an incidence structure $\cS$, a \textbf{tactical decomposition 
    of $\cS$} is a partition of the points of $\cS$ into point classes
  and the blocks of $\cS$ into block classes such that the number of
  points in a point class which are incident with a given block depends
  only on the class in which the block lies, and the number of blocks in
  a block class which are incident with a given point depends only on
  the class in which the point lies. 
\end{definition}
\begin{definition}
  Let $\bm{A} = \left[ a_{i j} \right]$ be a matrix, along with a
  partition of the row indices into subsets $R_{1}$, $\ldots$, $R_{t}$,
  and a partition of the column indices into subsets $C_{1}$, $\ldots$,
  $C_{t^{\prime}}$.  We will call this a \textbf{tactical decomposition
    of $\bm{A}$} if, for every $(i,j)$, $1 \le i \le t$, $1 \le j \le
  t^{\prime}$, the submatrix $\left[ a_{\bm{p}, \ell} \right]_{\bm{p}
    \in R_{i}, \ \ell \in C_{j}}$ has constant column sums $c_{i j}$ and
  row sums $r_{i  j}$.
\end{definition}  
The most common examples of tactical decompositions of an incidence
structure are obtained by taking as point and block classes the orbits
of some collineation group (although not all examples arise this
way). A tactical decomposition of an incidence structure corresponds
to a tactical decomposition of its incidence matrix.

The following result comes from the theory of the interlacing of
eigenvalues, which was introduced by Higman and Sims and further
developed by Haemers; see \cite{HaeEV} for a detailed survey.
\begin{theorem}
  Suppose $\bm{A}$ can be partitioned as 
  \begin{equation*}
    \bm{A} = 
    \left[ 
      \begin{array}{ccc}
        \bm{A}_{11}   & \cdots & \bm{A}_{1k} \\ 
        \vdots & \ddots & \vdots \\
        \bm{A}_{k1}   & \cdots & \bm{A}_{kk}
      \end{array}
    \right]
  \end{equation*}
  with each $\bm{A}_{ii}$ square, $1 \leq i \leq k$, and each
  $\bm{A}_{ij}$ having constant column sum $c_{ij}$.  Then any
  eigenvalue of the matrix $\bm{B} = [c_{ij}]$ is also an eigenvalue
  of $\bm{A}$. 
\end{theorem}
\begin{proof}
  We use an eigenvector of $\bm{B}$ to construct and eigenvector of
  $\bm{A}$ by expanding the vector. 
\end{proof}

\section{Characters and Gauss sums}
For some basics on characters of a finite field, see \cite{LidNied};
for an in depth look at Gauss sums, see \cite{GaussSums}.  We will
write $\rmi$ for $\sqrt{-1} \in \bbC$.
\begin{definition}
  Let $\bbF$ be a finite field.  A \textbf{character $\chi$} of $\bbF$
  is a group homomorphism from $\bbF^{*}$ to $\bbC^{*}$. We call 
  $\chi$ \textbf{trivial} if $\chi(x) = 1$ for all $x \in \bbF^{*}$.
\end{definition}

The characters of a finite field $\bbF$ form a group $\hat{\bbF}$
under multiplication, with $\hat{\bbF} \simeq \bbF^{*}$.  The
\textbf{order} of a character $\chi$ is the smallest integer $d$ for
which $\chi(x)^{d} = 1$ for all $x \in \bbF^{*}$. For any character
$\chi$ of $\bbF$, we will write $\overline{\chi}$ for the
conjugate character.  We will always extend characters to the entire
field $\bbF$ by defining $\chi(0) = 0$, except for the trivial
character $\chi_{1}$ for which we define $\chi_{1}(0) = 1$. 

\begin{definition}
  Consider a finite field $\bbF_{p^{r}}$ and let $\chi$ be a character
  of $\bbF_{p^{r}}$.  Let $\bbT_{r}: \bbF_{p^{r}} \to \bbF_{p}$ be the
  absolute trace function and let $\zeta = \exp{(2 \pi \rmi/p)}$.  The
  \textbf{Gauss sum} of $\chi$ is defined to be
  \[ 
  G_{r}(\chi) = \sum_{x \in \bbF_{p^{r}}}\chi(x) \zeta^{\bbT_{r}(x)}.
  \]
\end{definition}

We will need the following result, found in
\cite[Theorem~1.1.4~(a)]{GaussSums}: 
\begin{theorem} \label{GSumProd}
  Let $\chi$ be a nontrivial character of $\bbF_{p^{r}}$.  Then
  \[
  G_{r}(\chi) G_{r}(\overline{\chi}) = \chi(-1)p^{r}.
  \]
\end{theorem}

\begin{definition}
Let $p$ be odd, and fix a primitive element $\alpha$ of
$\bbF_{p^{r}}$.  We define the \textbf{quadratic character} $\chi_{2}$
of $\bbF_{p^{r}}$ by $\chi_{2}(0) = 0$, $\chi_{2}(\alpha^{n}) =
(-1)^{n}$. We have that $\chi_{2}$ is the unique character of
$\bbF_{p^{r}}$ having order $2$. 
\end{definition}

We apply the following result from \cite{GaussSums} to evaluate Gauss
sums of the quadratic character:
\begin{theorem}\label{GSum2Eval}
  Let $p$ be an odd prime, $\chi_{2}$ be the quadratic character on
  $\bbF_{p^{r}}$.  Then
  \[ 
  G_{r}(\chi_{2}) = 
  \begin{cases} 
    (-1)^{r-1}\sqrt{q} & \mbox{ if } p \equiv 1 \bmod{4}\\
    (-1)^{r-1} \rmi^{r} \sqrt{q} & \mbox{ if } p \equiv 3 \bmod{4}.
  \end{cases} 
  \]
\end{theorem}

\begin{theorem}[\textbf{Davenport-Hasse product
    relation}] \label{DavHasse} 
  Let $\psi$ be a character of $\bbF_{p^{r}}$ of order $d > 1$.  If
  $\chi$ is a nontrivial character of $\bbF_{p^{r}}$ such that $\chi
  \psi^{i}$ is nontrivial for all $1 \leq i < d$ then we have
  \[ 
  \frac{\chi^{d}(d) G_{r}(\chi)}{G_{r}(\chi^{d})} \prod_{i=1}^{d-1}
  \frac{G_{r}(\chi \psi^{i})}{G_{r}(\psi^{i})} = 1. 
  \]
\end{theorem}

\begin{lemma} \label{GSumRatio}
  Let $q = p^{h}$, and $\chi$ be a character of $\bbF_{q^{3}}$ such that 
  the restriction of $\chi$ to $\bbF_{q}$ is nontrivial.  Let $\rmT$
  denote the relative trace function $\bbF_{q^{3}} \to \bbF_{q}$. Then
  \begin{align}
    \label{GSR} \sum_{x \in \bbF_{q^{3}}} \overline{\chi}(\rmT(x)) \chi(x) & =
    (q-1)\frac{G_{3h}(\chi)}{G_{h}(\chi)}.
  \end{align}
\end{lemma}
\begin{proof}
  We begin by multiplying \eqref{GSR} by $G_{h}(\chi)$.  Note that we
  can omit the terms where $\rmT(x) = 0$, so
  \begin{align*}
    \sum_{\lambda \in \bbF_{q}} \chi(\lambda) \zeta^{\bbT_{h}(\lambda)} 
    \sum_{x \in \bbF_{q^{3}}} \overline{\chi}(\rmT(x)) \chi(x) & =
    \sum_{\lambda \in \bbF_{q}^{*}} \sum_{x \in \bbF_{q^{3}}, \rmT(x) \neq
      0} \overline{\chi}(\rmT(x))\chi(\lambda x) \zeta^{\bbT_{h}(x)}.
  \end{align*}
  We replace $\lambda$ with $\rmT(x) \mu$:
  \begin{align*}
    &= \sum_{\mu \in \bbF_{q}^{*}} \sum_{x \in \bbF_{q^{3}}, \rmT(x) \neq
      0} \overline{\chi}(\rmT(x)) \chi(\rmT(x) \mu x)
    \zeta^{\bbT_{h}(\rmT(x)\mu)}\\
    &= \sum_{\mu \in \bbF_{q}^{*}} \sum_{x \in \bbF_{q^{3}}, \rmT(x) \neq
      0} \chi(\mu x) \zeta^{\bbT_{h}(\rmT(\mu x))}.
  \end{align*}
  Now we replace $x$ by $\mu^{-1}y$ so that $\mu$ disappears from the
  formula: 
  \begin{align*}
    &= (q-1) \sum_{y \in \bbF_{q^{3}}, \rmT(y) \neq
      0} \chi(y) \zeta^{\bbT_{3h}(y)}.
  \end{align*}
  Except for the condition that $\rmT(y) \neq 0$, this is equal to
  $(q-1)G_{3h}(\chi)$.  However,
  \begin{align*}
    \sum_{y \in \bbF_{q^{3}}, \rmT(y) = 0} \chi(y) \zeta^{\bbT_{3h}(y)} &
    =  \sum_{y \in \bbF_{q^{3}}, \rmT(y) = 0} \chi(y).
  \end{align*}
  Let $\omega$ be a primitive element of $\bbF_{q}$, we have
  $\chi(\omega) \neq 1$ by assumption.  By replacing $y$ with $\omega
  y$, we get 
  \begin{align*}
    \sum_{y \in \bbF_{q^{3}}, \rmT(y) = 0} \chi(y) = \sum_{y \in
      \bbF_{q^{3}}, \rmT(y) = 0} \chi(\omega y) = \chi(\omega) \sum_{y \in
      \bbF_{q^{3}}, \rmT(y) = 0} \chi(y),
  \end{align*}
  which implies that the sum of the left hand side must be $0$ and that
  the condition $\rmT(y) \neq 0$ does not make a difference.
\end{proof}

\begin{definition}
  Let $f: \bbF^{*} \to \bbC$, and $\chi$ be any character of $\bbF$.  Then
  the \textbf{Fourier transform} of $f$ is defined to be
  \[ 
  \hat{f}(\chi) = \sum_{x \in \bbF^{*}} f(x) \overline{\chi}(x). 
  \]
\end{definition}

The Fourier transform is a bijective map from $\bbC^{\bbF^{*}} \to
\bbC^{\hat{\bbF}}$; that is, if $f : \bbF^{*} \to \bbC$, then
$\hat{f}: \hat{\bbF}\to \bbC$.\\

\section{Algebraic results}\label{SAlg}
Let $q$ be a prime power with $q \not \equiv 1 \bmod{3}$ and $q \equiv
1 \bmod{4}$.  We will consider the fields $\bbF =\bbF_{q}$ with
$\bbF^{*} = \langle \omega \rangle$, and $\bbE = \bbF_{q^{3}}$ with
$\bbE^{*} = \langle \alpha \rangle$, and assume that $\omega =
\alpha^{q^{2}+q+1}$.
\subsection{Preliminaries}
Define the following functions from $\bbE$ to $\bbF$:
\begin{align*}
  \rmT(x) = &x^{q^{2}} + x^{q} + x,\\ 
  \rmN(x) = &x^{q^{2} +q +1}.
\end{align*}

Through some simple computations, the following two properties can be
verified:
\begin{lemma} \label{TSymPoly}
  Let $x,y \in \bbE$.  Then
  \[ 
  (y+x)(y+x^{q})(y+x^{q^{2}}) = y^{3} + y^{2}\rmT(x) + y
  \rmT(x^{q+1})+ \rmN(x). 
  \]
  In the special case that $y \in \bbF$, the left hand side is equal to $\rmN(y+x)$.
\end{lemma}
\begin{lemma} \label{TTrx2}
  For any $x \in \bbE$, we have
  \[ 
  \rmT(x^{2}) = \rmT(x)^{2}-2\rmT(x^{q+1}). 
  \]
\end{lemma}

\begin{lemma} \label{TTrxqp1}
  Let $x \in \bbE$ such that $\rmT(x) = \rmT(x^{q+1}) = 0$. 
  \begin{enumerate}[label=\rm{(\roman*)}, ref=\rm{(\roman*)}]
  \item If $q \equiv 2 \bmod{3}$, then $x = 0$.
  \item If $q = 3^{h}$, then $x \in \bbF$.
  \end{enumerate}
\end{lemma}
\begin{proof}
  Applying Lemma~\ref{TSymPoly} with $y = -x$, we find $\rmN(x)=x^{3}$.
  Therefore $x^{3} \in \bbF$.  Since $q \not\equiv 1 \bmod{3}$, this
  implies that $x \in \bbF$.  It follows that $0 = \rmT(x) = 3x$.
\end{proof}

\begin{corollary}\label{TTrace}
  Let $x \in \bbE$ such that $\rmT(x) = \rmT(x^{2}) = 0$.
  \begin{enumerate}[label=\rm{(\roman*)}, ref=\rm{(\roman*)}]
  \item If $q \equiv 5 \bmod{6}$, then $x = 0$.
  \item If $q =3^{h}$, then $x \in \bbF$.
  \end{enumerate}
\end{corollary}
\begin{proof}
  This follows from Lemma~\ref{TTrx2} and Lemma~\ref{TTrxqp1}.
\end{proof}

\subsection{Cyclic model of the projective plane}\label{Smu}
Let $\mu = \alpha^{q-1}$, so $| \mu | = q^{2} +q +1$.  It is then
clear that $\rmN(\mu) = \mu^{q^{2} +q +1} = 1$.  

\begin{lemma}\label{muLI}
  The elements of $\langle \mu \rangle \subset \bbE^{*}$ are pairwise
  linearly independent over $\bbF$. 
\end{lemma}
\begin{proof}
  Assume otherwise, so that we have $\mu^{i} = \lambda \mu^{j}$, with
  $i \neq j \pmod{q^{2} +q +1}$, for some $\lambda \in \bbF$.
  Then 
  \[ 
  1 = \rmN(\mu^{i}) = \rmN(\lambda) \rmN(\mu^{j}) = \lambda^{3}. 
  \]
  Since $q \not \equiv 1 \bmod{3}$, this implies that $\lambda = 1$ 
  and so $\mu^{i} = \mu^{j}$, a contradiction.
\end{proof}

We can consider the elements of  $\langle \mu \rangle$ to represent
the points of a  projective plane $\pi$, with underlying vector space
$\bbE$ over $\bbF$.  The lines of $\pi$ correspond to solutions to the
equation $\rmT(\lambda x) = 0$ for a fixed $\lambda \in \bbE^{*}$,
which is an $\bbF$-linear map from $\bbE \to \bbF$.  The map $x
\mapsto \mu x$ gives a collineation of $\pi$ which  acts sharply
transitively on the points, as well as on the lines of $\pi$.

\begin{lemma}\label{SquareConic}
  The map $x \mapsto x^{2}$ induces a permutation on the points of
  $\pi$; the preimage of a line under this map is a conic.
\end{lemma}
\begin{proof}
  To see that this map is a permutation on the points of $\pi$, assume 
  that $\mu^{i}$ and $\mu^{j}$ get mapped to the same point, so
  $\mu^{2i} = \mu^{2j}$. This means that $2i \equiv 2j
  \bmod{q^{2}+q+1}$, which is odd.  This can only happen if $i \equiv
  j \bmod{q^{2}+q+1}$, in which case $\mu^{i} = \mu^{j}$. 

  Consider a line $\ell$ of $\pi$ described by the equation
  $\rmT(\lambda x) = 0$ for some $\lambda \in \bbE^{*}$.  The map $y
  \mapsto   \rmT(\lambda y^{2})$ defines a quadratic form; there are
  precisely $q+1$ isotropic points (the preimages of the points in
  $\ell$ under the map $x \mapsto x^{2}$) and so the isotropic points
  form a nondegenerate conic.
\end{proof}

\begin{lemma} 
  The bilinear form on $\bbE$ given by $(x,y) \mapsto \rmT(xy)$ is
  nondegenerate, and has square discriminant.
\end{lemma}
\begin{proof}
  There are $q^{2}$ elements of $\bbE$ with trace $0$, so there exists 
  an element of $\bbE^{*} \setminus \bbF^{*}$ having trace $0$; after 
  multiplying by a nonsquare in $\bbF^{*}$ if necessary, we can assume 
  that this element is a square.  So we have $v \in \bbE^{*} \setminus 
  \bbF^{*}$ with $\rmT(v^{2}) = 0$. By Corollary~\ref{TTrace}, we know
  that $\rmT(v) \neq 0$. Let $a = \lambda v$, where $\lambda$ is
  chosen so that $\rmT(a) = 2$.  We   have that $\rmT(a^{2}) = 0$, and
  it follows from Lemma~\ref{TTrx2} that $\rmT(a^{q+1}) = 2$.  Let $b
  = a^{q}$ and $c = 2-a-a^{q}$.  Then 
  \[
  \rmT(c^{2}) = 4\rmT(1)-4\rmT(a)+2\rmT(a^{q+1}) = 12-8-8+4=0. 
  \] 
  We can also compute
  \[ 
  \rmT(ac) = \rmT(bc) = 2\rmT(a)-\rmT(a^{q+1}) = 2. 
  \]
  The Gram matrix for our bilinear form is given by
  \[ 
  \begin{bmatrix}
    \rmT(a^{2}) & \rmT(ab) & \rmT(ac)\\
    \rmT(ab) & \rmT(b^{2}) & \rmT(bc)\\
    \rmT(ac) & \rmT(bc) & \rmT(c^{2})
  \end{bmatrix} 
  = \begin{bmatrix}
    0 & 2 & 2\\
    2& 0 & 2\\
    2& 2& 0
  \end{bmatrix}, 
  \]
  which has determinant $16$.
\end{proof}

\begin{corollary}\label{Discabc}
  Let $a,b,c \in \bbE^{*}$ be distinct, with
  $\rmN(a)=\rmN(b)=\rmN(c)=1$ and $\rmT(a^{2}) = \rmT(b^{2}) =
  \rmT(c^{2}) = 0$.  Then $2\rmT(ab)\rmT(ac)\rmT(bc)$ is always a
  nonzero square.  
\end{corollary}
\begin{proof}
  We have that $a^{2}$, $b^{2}$ and $c^{2}$ are collinear, so by
  Lemma~\ref{SquareConic}, $a$, $b$, and $c$ are noncollinear; this
  implies that $\{a, b, c\}$ forms a basis for $\bbE$ over $\bbF$.
  Therefore we have
  \[  
  \begin{vmatrix}
    \rmT(a^{2}) & \rmT(ab) & \rmT(ac)\\
    \rmT(ab) & \rmT(b^{2}) & \rmT(bc)\\
    \rmT(ac) & \rmT(bc) & \rmT(c^{2})
  \end{vmatrix} = 2\rmT(ab) \rmT(ac) \rmT(bc)
  \]
  is a square, and nonzero since the form is nondegenerate.
\end{proof}

\subsection{Results on $\kappa(x)$}\label{Skappa}
Let $\chi_{4}$ be the \textbf{quartic character} of $\bbE$; this is
defined by $\chi_{4}(0) = 0$, $\chi_{4}(\alpha^{n}) = \rmi^{n}$ (and
is dependent on our choice of a primitive element $\alpha$).  Notice 
that $\chi_{4}(x)^{2} = \chi_{2}(x)$, where $\chi_{2}$ is the
quadratic character of $\bbE$.  We will assume that $\omega =
\alpha^{-(q^{2}+q+1)}$, so that $\chi_{4}(\omega) =
\rmi^{-(q^{2}+q+1)} = \rmi$.
\begin{definition}\label{Dkappa}
  For $z \in \{ 1, \rmi, -1, -\rmi \}$, we define functions
  $\kappa_{z}  : \bbE^{*} \to \bbZ$ by 
  \[ 
  \kappa_{z}(x) = |\{i : 0 \leq i < q^{2} +q +1 \ | \
  \chi_{4}(\rmT(\mu^{i}x))\overline{\chi_{4}}(\rmT(\mu^{-i}x)) =z\}|. 
  \]
\end{definition}
\begin{lemma}
  For any $\lambda \in \bbF^{*}$, $z \in \{ 1, \rmi, -1, -\rmi \}$, we
  have $\kappa_{z}( \lambda x) = \kappa_{z}(x).$
\end{lemma}
\begin{proof}
  Since $\chi_{4}(\lambda) \overline{\chi}_{4}(\lambda) = 1$ we have
  that, for $0 \leq i < q^{2} +q +1$,  
  \begin{align*} 
    \chi_{4}(\rmT(\mu^{i} \lambda x))\overline{\chi_{4}}(\rmT(\mu^{-i}
    \lambda x)) & = \chi_{4}(\lambda) \chi_{4}(\rmT(\mu^{i}x))
    \overline{\chi}_{4}(\lambda) \overline{\chi_{4}}(\rmT(\mu^{-i}x))\\ 
    & = \chi_{4}(\rmT(\mu^{i}x))\overline{\chi_{4}}(\rmT(\mu^{-i}x)). 
  \end{align*}
\end{proof}

\begin{lemma} \label{kappazzbar}
  For every $x \in \bbE^{*}$, $\kappa_{i}(x) = \kappa_{-i}(x)$. 
\end{lemma}
\begin{proof}
  Putting $j^{\prime} = q^{2} +q +1 -j$, we have
  \[
  \chi_{4}(\rmT(\mu^{j}x))\overline{\chi_{4}}(\rmT(\mu^{-j}x))  
  = \chi_{4}(\rmT(\mu^{-j^{\prime}}x))\overline{\chi_{4}}(\rmT(\mu^{j^{\prime}}x)),
  \]
  so the values of $j$ for which
  \[
  \chi_{4}(\rmT(\mu^{j}x))\overline{\chi_{4}}(\rmT(\mu^{-j}x)) =i
  \]
  are in one-to-one correspondence with the values of $j^{\prime}$ for
  which
  \[
  \chi_{4}(\rmT(\mu^{j^{\prime}}x))\overline{\chi_{4}}(\rmT(\mu^{-j^{\prime}}x))
  = -i. 
  \]
\end{proof}

\begin{theorem} \label{kappa1minus}
  For every $x \in \bbE^{*}$, we have
  \[
  \kappa_{1}(x) - \kappa_{-1}(x) = q \cdot \chi_{2}(x)
  \chi_{2}(\rmT(x)). 
  \]
\end{theorem}
\begin{proof}
  Define
  \begin{align*}
    A(x) & = (q-1) \chi_{2}(x) (\kappa_{0}(x) -\kappa_{2}(x)),\\  
    B(x) & = q (q - 1) \chi_{2}(\rmT(x)).
  \end{align*}
  We must show that $A(x) = B(x)$ for all $x \in \bbE^{*}$.
  
  We rewrite the expression for $A(x)$ using $\chi_{4}$:
  \begin{align*}
    A(x)
    &= (q - 1) \chi_{2}(x) (\kappa_{0}(x) + \rmi
    \kappa_{1}(x) - \kappa_{2}(x) - \rmi \kappa_{3}(x)) \\ 
    &= (q - 1) \sum_{\rmN(a) = 1} \chi_{4}(\rmT(a x))
    \overline{\chi_{4}}(\rmT(a\inv x)) \chi_{2}(x). 
  \end{align*}
  We write $q - 1$ as $\sum_{\lambda \in \bbF^{*}} 1$ and use the fact that
  elements of norm $1$ are always squares:
  \begin{align*}
    A(x)
    &= \sum_{\lambda \in \bbF^{*}} \sum_{\rmN(a) = 1}
    \chi_{4}(\rmT(a x)) \overline{\chi_{4}}(\rmT(a\inv x))
    \chi_{2}(x) \\ 
    &= \sum_{\lambda \in \bbF^{*}} \sum_{\rmN(a) = 1}
    \chi_{4}(\lambda \rmT(a x))
    \overline{\chi_{4}}(\lambda\inv \rmT(a\inv x))
    \chi_{4}(\lambda\inv) \overline{\chi_{4}}(\lambda)
    \chi_{2}(a\inv x) \\ 
    &= \sum_{\lambda \in \bbF^{*}} \sum_{\rmN(a) = 1}
    \chi_{4}(\rmT(\lambda a x))
    \overline{\chi_{4}}(\rmT(\lambda\inv a\inv x))
    \chi_{2}(\lambda\inv a\inv x). 
  \end{align*}
  If $\lambda$ runs over all elements of $\bbF^{*}$ and $a$ over all
  elements of $\bbE^{*}$ of norm $1$, then $\lambda a$ runs over all
  elements of $\bbE^{*}$: 
  \[
  A(x) = \sum_{a} \chi_{4}(\rmT(a x))
  \overline{\chi_{4}}(\rmT(a\inv x)) \chi_{2}(a\inv x).  
  \]
  We now substitute $a = b\inv x$:
  \[
  A(x) = \sum_{b} \chi_{4}(\rmT(b\inv x^{2}))
  \overline{\chi_{4}}(\rmT(b)) \chi_{2}(b). 
  \]
  Let $\chi$ be any character of $\bbE$.
  Then the Fourier transforms of $A$ and $B$ are defined to be 
  \begin{align*}
    \hat{A}(\chi) &= \sum_{x} A(x) \overline{\chi}(x) \\
    &= \sum_{x} \sum_{b} \chi_{4}(\rmT(b\inv x^{2}))
    \overline{\chi_{4}}(\rmT(b)) \chi_{2}(b)
    \overline{\chi}(x), \\ 
    \hat{B}(\chi) &= \sum_{x} B(x) \overline{\chi}(x) \\
    &= q (q-1) \sum_{x} \chi_{2}(\rmT(x))
    \overline{\chi}(x).
  \end{align*}
  Since the Fourier transform is a $1$-to-$1$ operation, we are reduced
  to showing that $\hat{A}(\chi) = \hat{B}(\chi)$ for all characters
  $\chi$ of $\bbE$. 
  
  Let $\omega = \alpha^{q^{2} + q + 1}$ be a generator of $\bbF^*$. 
  If we replace $x$ by $\omega x$ in the formula for $\hat{A}(\chi)$
  and we use $A(\omega x) = \chi_2(\omega) A(x)$, we get
  \begin{align*}
    \hat{A}(\chi) &= \sum_{x} A(\omega x) \overline{\chi}(\omega x) \\
    &= \sum_{x} \chi_{2}(\omega) \overline{\chi}(\omega)
    A(x) \overline{\chi}(x) = \chi_{2}(\omega)
    \overline{\chi}(\omega) \hat{A}(\chi). 
  \end{align*}
  If $\chi_{2}(\omega) \overline{\chi}(\omega) \neq 1$, this implies
  that $\hat{A}(\chi) = 0$.  Analogously, also $\hat{B}(\chi) = 0$ in
  this case.  We can therefore restrict ourselves to the case where 
  $\chi_{2}(\omega) \overline{\chi}(\omega) = 1$.  Since
  $\chi_{2}(\omega) = -1$, this means $\overline{\chi}(\omega) =
  \overline{\chi}(\alpha)^{q^{2} + q +1} = -1$. It follows that
  $\overline{\chi}^{q^{2} + q + 1} = \chi_{2}$.
  
  Since the restriction of $\overline{\chi}$ to $\bbF$ equals
  $\chi_{2}$, we can evaluate $\hat{B}(\chi)$ using
  Lemma~\ref{GSumRatio}:  
  \begin{equation}\label{Bhat}
    \hat{B}(\chi) = q (q-1)^{2}
    \frac{G_{3h}(\overline{\chi})}{G_{h}(\chi_{2})} 
  \end{equation}
  
  The order of $\overline{\chi}$ is a divisor of $2(q^2 + q + 1)$.
  Since the character group is cyclic of order $q^{3} - 1$, which is a
  multiple of $4(q^{2} + q + 1)$, it follows that $\overline{\chi} =
  \sigma^{2}$ for some character $\sigma$.  We can choose $\sigma$ such
  that $\sigma^{q^{2}+q+1} = \chi_{4}$.  Since $3(q^{2}+q+1) \equiv 1
  \bmod{4}$, the restriction of $\sigma$ to $\bbF$ is $\chi_{4}^{3} =
  \overline{\chi}_{4}$.  (\textbf{remark:} This makes $\chi =
  \overline{\sigma}^{2}$.)  
  
  For $\hat{A}(\chi)$, we get
  \begin{align*}
    \hat{A}(\chi) = \hat{A}(\overline{\sigma}^{2})
    &= \sum_{x} \sum_{b} \chi_4(\rmT(b\inv x^{2}))
    \overline{\chi_{4}}(\rmT(b)) \chi_2(b) \sigma(x^{2}) \\ 
    &= 2 \sum_{s \text{ a square}} \sum_{b}
    \chi_{4}(\rmT(b\inv s)) \overline{\chi_{4}}(\rmT(b))
    \chi_{2}(b) \sigma(s). 
  \end{align*}
  If we replace $s$ by $\omega s$ in the above sum, the terms remain all
  invariant.  This operation exchanges squares and non-squares.
  Therefore, $\hat{A}(\sigma^{2})$ is equal to 
  \[
  \hat{A}(\overline{\sigma}^{2})
  = \sum_{s} \sum_{b} \chi_{4}(\rmT(b\inv s))
  \overline{\chi_{4}}(\rmT(b)) \chi_{2}(b) \sigma(s). 
  \]
  After replacing $s$ by $b y$, we get
  \begin{align*}
    \hat{A}(\overline{\sigma}^{2})
    &= \sum_{y} \sum_{b} \chi_{4}(\rmT(y))
    \overline{\chi}_{4}(\rmT(b)) \chi_{2}(b) \sigma(b)
    \sigma(y) \\ 
    &= \left(\sum_{y} \chi_{4}(\rmT(y)) \sigma(y) \right)
    \left(\sum_{b} \overline{\chi}_{4}(\rmT(b))
      \chi_{2}(b) \sigma(b) \right). 
  \end{align*}

  Let $\tau = \sigma \cdot \chi_{2}$, then the restriction of $\tau$ to
  $\bbF$ is $\overline{\chi}_{4} \chi_{2} = \chi_{4}$. 
  
  \begin{align*}
    \hat{A}(\overline{\sigma}^{2})
    &= \left(\sum_{y} \overline{\sigma}(\rmT(y)) \sigma(y)\right) 
    \left(\sum_{b} \overline{\tau}(\rmT(b)) \tau(b)\right). 
  \end{align*}

  We evaluate this using Lemma~\ref{GSumRatio}:
  \begin{align*}
    \hat{A}(\overline{\sigma}^{2})
    &= (q-1)^{2} \frac{G_{3h}(\sigma) G_{3h}(\tau)}{G_h(\sigma)
      G_h(\tau)} 
    = (q-1)^{2} \frac{G_{3h}(\sigma)
      G_{3h}(\tau)}{G_{h}(\overline{\chi_{4}}) G_{h}(\chi_{4})}. 
  \end{align*}
  Applying Theorem~\ref{GSumProd} gives
  \[
  \hat{A}(\overline{\sigma}^{2})
  = q\inv (q-1)^{2} \chi_{4}(-1) G_{3h}(\sigma)
  G_{3h}(\tau). 
  \]

  Finally, we apply the Davenport--Hasse product formula
  (Theorem~\ref{DavHasse}) which states that 
  \begin{equation}
    \frac{\sigma^{2}(2) G_{3h}(\sigma) G_{3h}(\sigma
      \chi_{2})}{G_{3h}(\sigma^{2}) G_{3h}(\chi_{2})} = 1 
  \end{equation}
  Replacing $\overline{\sigma}^{2}$ back by $\chi$, we get
  \begin{align*}
    \hat{A}(\chi)
    &= q\inv (q-1)^{2} \chi_{4}(-1) \overline{\chi}(2) G_{3h}(\chi)
    G_{3h}(\chi_{2}) \\ 
    &= q\inv (q-1)^{2} \chi_{4}(-1) \chi_{2}(2) G_{3h}(\chi_{2})
    G_{3h}(\chi). 
  \end{align*}

  Since $(1 + \rmi)^{4} = -4$ and $\bbF$ contains a square root of
  $-1$, it follows that $\chi_{4}(-1) \chi_{2}(2) = \chi_{4}(-4) = 1$:
  \[
  \hat{A}(\chi) = q\inv (q-1)^2 G_{3h}(\chi_2) G_{3h}(\chi).
  \]

  Dividing this by \eqref{Bhat} gives:
  \[
  \frac{\hat{A}(\chi)}{\hat{B}(\chi)} = \frac{1}{q^{2}}
  G_{3h}(\chi_{2}) G_{h}(\chi_{2}). 
  \]
  The explicit formulas for quadratic Gauss sums given in
  Theorem~\ref{GSum2Eval} show that this always equals
  $1$ if $q \equiv 1 \bmod{4}$, hence $\hat{A}(\chi) =
  \hat{B}(\chi)$.  
\end{proof}

\section{Our setup}
\subsection{A model for $\cQ^{+}(5,q)$}
As in Section~\ref{SAlg}, we have $q = p^{h}$ with $q \equiv 1
\bmod{4}$ and $q \not\equiv 1 \bmod{3}$ (so $q \equiv 5 \mbox{ 
  or } 9 \bmod{12}$).  We again define $\bbF=\bbF_{q}$ and $\bbE =
\bbF_{q^{3}}$ with primitive elements $\omega$ and $\alpha$,
respectively (where $\omega$ is chosen to be $\alpha^{q^{2}+q+1}$),
and the relative trace and norm functions from $\bbE \to \bbF$ by 
\begin{align*}
\rmT(x) & = x^{q^{2}} + x^{q} + x,\\ 
\rmN(x) & = x^{q^{2} +q +1}.
\end{align*}

Let $\pg(5,q)$ have the underlying $\bbF$-vector space $V =
\bbE^{2}$, and consider the quadratic form $\rmQ$ on $V$ given by
\[
\rmQ((u,v)) = \rmT(uv).
\]
The polar form $\rmB$ of $\rmQ$ is then given by 
\[
\rmB((u_{1},v_{1}), (u_{2}, v_{2})) = \rmT(u_{1}v_{2}) +
\rmT(v_{1}u_{2}).
\]
\begin{lemma} \label{QDef}
  The form $\rmQ$ defined above is nondegenerate, and the associated
  quadric is a $\cQ^{+}(5,q)$.
\end{lemma}
\begin{proof} 
  If there exists $(u, v) \in V$ with $\rmB((u, v), (x,y)) = 0$ for
  all $(x,y) \in V$, then $\rmT(uy) + \rmT(vx) = 0$ for all $(x,y) \in
  V$.  Setting $x = 0$ forces us to have $\rmT(uy) = 0$ for all $y \in
  \bbE^{*}$, thus $u = 0$.  Likewise setting $y = 0$ can be seen to
  force $v = 0$, and so $(u, v) = (0,0)$. Thus $\rmQ$ is
  nondegenerate.  We can see that $\{ (x,0) \ : \ x \in \bbE \}$ is a
  totally singular subspace of $V$ with dimension $3$, so $(V,\rmQ)$
  has maximal Witt dimension. Therefore we have that $\rmQ$ is
  hyperbolic, with the associated quadric being a $\cQ^{+}(5,q)$.
\end{proof} 

\subsection{A group}
As in Section~\ref{Smu}, we put $\mu = \alpha^{q-1}$, so $| \mu | =
q^{2} +q +1$.  Recall that $\rmN(\mu)=1$, and that the elements of
$\langle \mu \rangle$ are pairwise $\bbF$-linearly independent over
$\bbE$. 

\begin{lemma}\label{LcDef}
  The map $c : (u,v) \mapsto (\mu u,  \mu^{-1} v)$ is a projective
  isometry of $\cQ^{+}(5,q)$, and $\langle c \rangle$ acts
  semi-regularly on the points of the space. There are $q^{2} +1$
  orbits of $\langle c \rangle$ on $\cQ^{+}(5,q)$; these include the
  generators $\pi_{1}$ and $\pi_{2}$ given by
\begin{align*}
  \pi_{1} & = \{(x,0) : \ x \in \bbE^{*}\}\mbox{ and}\\ 
  \pi_{2} & = \{(0,y) : \  y \in \bbE^{*}\}.
\end{align*}
A collection of representatives for the remaining $(q^{2}-1)$
orbits is given by  
\[
\{ (1, \lambda a^{2}) : \ a \in \cS, \lambda \in \bbF^{*}\},
\]
where 
\[
\cS = \{ a : \ a \in \bbE^{*} \ | \ \rmN(a) = 1, \rmT(a^{2})=0\}. 
\]  
\end{lemma}
\begin{proof}
  It is easy to see that  $c$ is an isometry.  Now, if $c^{i}$ fixes
  some point $(u,v) \in \cQ^{+}(5,q)$, then $(\mu^{i}u, \mu^{-i}v) =
  (\lambda u, \lambda v)$ for some $\lambda \in \bbF^{*}$.  But by
  Lemma~\ref{muLI}, we cannot have $\mu^{i} \in \bbF^{*}$ unless
  $i=0$. So $\langle c \rangle$ acts semi-regularly on the points of
  $\cQ^{+}(5,q)$ and so must have $q^{2}+1$ orbits since each orbit
  has size $q^{2}+q+1$ and there are $(q^{2}+1)(q^{2}+q+1)$ points in
  total in the space. 

  Now $\pi_{1}$ and $\pi_{2}$ are each clearly point orbits of
  $\langle c \rangle$.  Any point $\bm{p} \in \cQ^{+}(5,q) \setminus 
  (\pi_{1} \cup \pi_{2})$ has the form $(u,v)$, with $u,v \in
  \bbE^{*}$ satisfying $\rmT(uv)=0$, and we can multiply by a scalar
  in $\bbF$ to assume that $\rmN(u) = 1$; this implies that $u =
  \mu^{i}$ for some $0 \leq i < q^{2}+q+1$.  We can also write $v =
  \lambda \mu^{j}$ for some $\lambda \in \bbF^{*}$ and $0 \leq j <
  q^{2}+q+1$.  Then $(u,v)$ maps under $c^{-i}$ to $(1, \lambda
  \mu^{i+j})$, with $\rmT(\mu^{i+j})=0$. Since $\langle \mu \rangle$
  is a cyclic group with odd order, there is an $a \in \langle \mu
  \rangle$ with $a^{2} = \mu^{i+j}$, therefore $(u,v)$ is in the same
  orbit under $\langle c \rangle$ as $(1, \lambda a^{2})$, with
  $\rmN(a) = 1$ and $\rmT(a^{2}) = 0$.   

  To see that, for $a, a^{\prime} \in \cS$, $(1, \lambda a^{2})$ and
  $(1, \lambda^{\prime} a^{\prime 2})$ are not in the same orbit under
  $\langle c \rangle$ unless $\lambda = \lambda^{\prime}$ and $a =
  a^{\prime}$, we recall that $\mu^{i} \in \bbF^{*}$ implies that
  $i=0$.  Therefore we would have to have $\lambda a^{2} =
  \lambda^{\prime} a^{\prime 2}$.  But since $\rmN(a) =
  \rmN(a^{\prime}) =1$, we must have $\rmN(\lambda) =
  \rmN(\lambda^{\prime})$ or equivalently, $\lambda^{3} = 
  \lambda^{\prime 3}$.  Since $q \not\equiv 1 \bmod{3}$, this
  implies that $\lambda = \lambda^{\prime}$ and $a^{2} = a^{\prime
    2}$, thus $a = \pm a^{\prime}$.  But since $\rmN(a) =
  \rmN(a^{\prime}) = 1$, we must have $a = a^{\prime}$. 
\end{proof}

\begin{lemma} \label{GrpOrbs}
  The map $z : (u,v) \mapsto (u, \omega^{4} v)$ is a  projective
  similarity of $\cQ^{+}(5,q)$ which centralizes $c$.  Define the
  group $G = \langle c,z \rangle$.  Then $G$ has $2+4(q+1)$ point
  orbits on $\cQ^{+}(5,q)$, including the planes $\pi_{1}$ and
  $\pi_{2}$.  The remaining $4(q+1)$ orbits each have size
  $\frac{(q-1)}{4}(q^{2} +q +1)$, and are represented by the points
  \[
  \{(1,\omega^{s} a^{2}) : \ 0 \leq s < 4, \  a  \in \cS \},
  \]
  where $\cS$ is as defined in Lemma~\ref{LcDef}.
\end{lemma}
\begin{proof}
  It is clear that $z$ is a projective similarity of order $(q-1)/4$,
  and that $z$ commutes with $c$, therefore $z$ permutes the orbits of
  $\langle c \rangle$.  Furthermore, $z$ stabilizes the planes
  $\pi_{1}$ and $\pi_{2}$ pointwise.  

  Now consider two orbits $\cO_{1}$, $\cO_{2}$ of $\langle c \rangle$
  on $\cQ^{+}(5,q) \setminus (\pi_{1} \cup \pi_{2})$ represented by
  points $(1, \lambda a^{2})$ and $(1, \lambda^{\prime} a^{\prime
    2})$, respectively, with $a, a^{\prime} \in \cS$ and $\lambda,
  \lambda^{\prime} \in \bbF^{*}$.  If $\cO_{1} \mapsto \cO_{2}$ under
  $z^{k}$ for some $0 \leq k < (q-1)/4$, then $(1, \omega^{4k}\lambda
  a^{2}) \in \cO_{2}$.  This can only happen if $\omega^{4k}\lambda
  a^{2} = \lambda^{\prime} a^{\prime 2}$.  Since $\rmN(a^{2}) =
  \rmN(a^{\prime 2}) = 1$, we must have $\rmN(\omega^{4k}\lambda) =
  \rmN(\lambda^{\prime})$; but $x \mapsto \rmN(x)$ is a permutation of
  $\bbF$ since $q \not\equiv 1 \bmod{3}$, so $\lambda^{\prime} =
  \omega^{4k}\lambda$ for some $0 \leq k < (q-1)/4$ and $a^{\prime} =
  a$.  Therefore the orbit under $G$ containing the point $(1, \lambda
  a^{2})$, with $a \in \cS$, can be described by writing $\lambda =
  \omega^{4k+s}$; then the point orbit in question has unique
  representative $(1, \omega^{s} a^{2})$.  The number of orbits
  follows immediately. 
\end{proof}

\subsection{A tactical decomposition}
As we are interested in finding eigenvectors of the collineation
matrix $\bm{A}$ of $\cQ^{+}(5,q)$, we will define a useful tactical
decomposition of this matrix.  If we think of $\bm{A}$ as the
incidence matrix of the structure whose ``points'' and ``blocks'' are
both the set of points of $\cQ^{+}(5,q)$ with incidence being given by
collinearity in the quadric, then it is clear that any collineation
group of $\cQ^{+}(5,q)$ will induce a tactical decomposition of
$\bm{A}$.  

Let 
\begin{equation}\label{eqa}
  \{ a_{1}, \ \ldots, a_{q+1} \} = \{ a : \ a \in \bbE^{*} \ | \rmN(a)
  = 1, \ \rmT(a^{2}) = 0\}.
\end{equation}
\begin{remark}
It can be seen from Lemma~\ref{LcDef} that the size
of the set on the right hand side of \ref{eqa} is in fact $q+1$.
\end{remark}
We wish to consider the tactical decomposition of $\bm{A}$ induced by
the group $G$ defined in Lemma~\ref{GrpOrbs}.  As seen in
Lemma~\ref{GrpOrbs}, the orbits of $G$ on $\cQ^{+}(5,q) \setminus
(\pi_{1} \cup \pi_{2})$ are represented uniquely by the points $(1, \
w^{s} a^{2}_{i})$ for $0 \leq s < 4$, $1 \leq i \leq q+1$.  Order
these points 
\begin{equation}\label{eqOrbitOrder}
  \begin{array}{lll}
    (1,  a^{2}_{1}), & \ldots, &  (1,  a^{2}_{q+1}),\\
    (1, \omega a^{2}_{1}), & \ldots, & (1, \omega a^{2}_{q+1}), \\
    (1, \omega^{2} a^{2}_{1}), & \ldots, & (1, \omega^{2} a^{2}_{q+1}),\\
    (1, \omega^{3} a^{2}_{1}), & \ldots, & (1, \omega^{3} a^{2}_{q+1}).
  \end{array}
\end{equation}

Let $\bm{B}$ be the column sum matrix associated with the tactical
decomposition of $\cQ^{+}(5,q)$ induced by the group $G$, after
throwing away the rows and columns corresponding to points in
$(\pi_{1} \cup \pi_{2})$ (as in Theorem~\ref{EvecMod}).

\begin{definition}
For $a, b \in \bbE^{*}$ with $\rmT(a) = \rmT(b) = 0$, and $0 \leq s \leq 3$,
define  
\begin{align*}
\kappa_{s}(a,b)  & := \left| \{ (k,i) : 0 \leq k < (q-1)/4, \ 0 \leq i
  < q^{2} +q +1 \ | \ \rmT(\mu^{i}a) = -\omega^{4k+s}\rmT(\mu^{-i}b)
  \}  \right| 
\end{align*}
\end{definition}

\begin{lemma} \label{TDMat}
The matrix $\bm{B}$ can be described as follows: for $0 \le s \le 3$, let
$\bm{B}_{s} = \left[ \kappa_{s}(a_{i}^{2}, a_{j}^{2})  \right]_{1 \le
  i,j \le q+1}$, where the $a_{i}$ are as described in
Equation~\ref{eqa}. Then   
\begin{equation}\label{eqB}
  \bm{B} = \begin{bmatrix}
    (\bm{B}_{0}-\bm{I}) & \bm{B}_{1} & \bm{B}_{2} & \bm{B}_{3} \\ 
    \bm{B}_{3} & (\bm{B}_{0}-\bm{I}) & \bm{B}_{1} & \bm{B}_{2} \\
    \bm{B}_{2} & \bm{B}_{3} & (\bm{B}_{0}-\bm{I}) & \bm{B}_{1} \\ 
    \bm{B}_{1} & \bm{B}_{2} & \bm{B}_{3} & (\bm{B}_{0}-\bm{I})
  \end{bmatrix} = \begin{bmatrix}
    \bm{B}_{0}^{\prime} & \bm{B}_{1} & \bm{B}_{2} & \bm{B}_{3} \\ 
    \bm{B}_{3} & \bm{B}_{0}^{\prime} & \bm{B}_{1} & \bm{B}_{2} \\
    \bm{B}_{2} & \bm{B}_{3} & \bm{B}_{0}^{\prime} & \bm{B}_{1} \\ 
    \bm{B}_{1} & \bm{B}_{2} & \bm{B}_{3} & \bm{B}_{0}^{\prime}
  \end{bmatrix}.
\end{equation}
\end{lemma}
\begin{proof}
This can be seen by observing the relationship between the functions
$\kappa_{z}$ and the polar form $\rmB$ of the quadratic form $\rmQ$
associated with $\cQ^{+}(5,q)$.
\end{proof}

\subsection{Eigenvectors of $\bm{B}$}
We have that $\bm{B}$ is real symmetric and block-circulant.
Eigenvectors of block-circulant matrices have been studied in
\cite{GTee}, and can be found using the following result:
\begin{theorem}\label{BCirc}
Let $\zeta$ be a fourth root of unity in $\bbC$, and let 
\[
\bm{A} = \begin{bmatrix}
\bm{A}_{0} & \bm{A}_{1} & \bm{A}_{2} & \bm{A}_{3}\\
\bm{A}_{3} & \bm{A}_{0} & \bm{A}_{1} & \bm{A}_{2}\\
\bm{A}_{2} & \bm{A}_{3} & \bm{A}_{0} & \bm{A}_{1}\\
\bm{A}_{1} & \bm{A}_{2} & \bm{A}_{3} & \bm{A}_{0}
\end{bmatrix}.
\]
Then if $\bm{v}$ is an eigenvector of $\bm{H} = \sum_{s = 0}^{3}
\zeta^{s} \bm{A}_{s}$, then
\[
\bm{w} = 
\begin{bmatrix} 
  \zeta^{0} \bm{v} \\ 
  \zeta^{1} \bm{v} \\
  \zeta^{2} \bm{v} \\
  \zeta^{3} \bm{v}
\end{bmatrix}
\]
is an eigenvector of $\bm{A}$.
\end{theorem}

Putting $\zeta = \rmi$, we will find eigenvectors of $\bm{B}$ by
considering the matrix
\begin{equation}\label{eqHdef}
  \bm{H} = \bm{B}_{0}^{\prime} + \sum_{s=1}^{3} \rmi^{s} \bm{B}_{s} =
  \left( \sum_{s=0}^{3} \rmi^{s} \bm{B}_{s} \right) - \bm{I}. 
\end{equation}
The entries $\kappa_{s}(a_{i}^{2}, a_{j}^{2})$ of $\bm{B}_{s}$ for $0
\leq s < 3$ can be analyzed by the following relationship to the
functions $\kappa_{z}(x)$ for $z \in \{1, \rmi, -1, -\rmi \}$
given in Definition~\ref{Dkappa}.

\begin{theorem}
  Let $a, b \in \bbE^{*}$ with $\rmN(a) = \rmN(b) = 1$ and $\rmT(a^{2})
  = \rmT(b^{2}) = 0$. Then
  \begin{align*}
    \kappa_{s}(a^{2},b^{2}) &= \kappa_{\rmi^{s}\chi_{2}(2)}(ab)  + \left( \frac{q-1}{4}
    \right) \epsilon_{ab} \\
  \end{align*}
  where 
  \[
  \epsilon_{ab} = \left| \{ j : \ \rmT(\mu^{j} ab) = \rmT(\mu^{-j} ab)
    = 0 \} \right|. 
  \]
\end{theorem}
\begin{proof}
  Take $a$ and $b$ as above.  Then $\kappa_{s}(a^{2}, b^{2})$ counts the
  number of pairs $(k,i)$, $0 \leq k < (q-1)/4$, $0 \le i < q^{2} +q+1$,
  such that 
  \[
  \rmT(\mu^{i}a^{2}) +\omega^{4k+s} \rmT(\mu^{-i}b^{2}) = 0.
  \]
  Since $\rmN(ab^{-1}) = 1$, there exists $n$ with $0 \leq n < q^{2} +q
  +1$ such that $\mu^{n} = ab^{-1}$. Then $\mu^{-(n-i)}a^{2} =
  \mu^{i}ab$ and $\mu^{(n-i)}b^{2} =  \mu^{-i}ab$, so we can replace
  $i$ with $(i-n)$ so that we are instead counting pairs $(k,i)$
  satisfying 
  \begin{equation}\label{Tkeq}
  \rmT(\mu^{i} ab) = - \omega^{4k+s} \rmT(\mu^{-i} ab).
  \end{equation}

  Now, consider a pair $(k,i)$ satisfying this equation.  If
  $\rmT(\mu^{-i}ab) =0$, then we must also have $\rmT(\mu^{i}ab) = 0$, 
  and the choice of $k$ is immaterial.  Thus there are
  $\left(\frac{q-1}{4}\right) \epsilon_{ab} $ such pairs.  If
  $\rmT(\mu^{-i}ab) \neq 0$ then, applying $\chi_{4}$ to both sides of 
  Equation~\ref{Tkeq}, we have
  \[
  \chi_{4}(\rmT(\mu^{i} ab)) = \chi_{4}(-
  \omega^{4k+s})\chi_{4}(\rmT(\mu^{-i} ab)),
  \]
  and since $\chi_{4}(\rmT(\mu^{-i} ab)) \neq 0$,
  \[
  \chi_{4}(\rmT(\mu^{i} ab)) \overline{\chi}_{4}(\rmT(\mu^{-i} ab)) = \chi_{4}(-
  \omega^{4k+s}) = \rmi^{s} \chi_{4}(-1).
  \]
  Now since $q \equiv 1 \bmod{4}$, $\chi_{4}(-1) = \pm 1$, and 
  $\chi_{4}(-1) = 1$ if and only if $q \equiv 1 \bmod{8}$.

  Notice that, with $q = p^{h}$ for some prime $p$, if $h$ is even then 
  $q \equiv 1 \bmod{8}$ and every element of $\bbF_{p}$ is a square in 
  $\bbF_{q}$ so $\chi_{2}(2) = 1$.  On the other hand if $h$ is odd,
  $\bbF_{q}$ is an odd-degree field extension of $\bbF_{p}$, so $2$ is
  a square in $\bbF_{q}$ if and only if it is a square in $\bbF_{p}$.
  Since $h-1$ is even, $q = p \cdot p^{h-1} \equiv p \bmod{8}$, thus
  $\chi_{2}(2) = (-1)^{\frac{p^{2}-1}{8}} = 1$ if and only if $q
  \equiv 1 \bmod{8}$.  In either case, $\chi_{4}(-1) = \chi_{2}(2)$.
\end{proof}

\begin{corollary}\label{ThBig}
  Let $a, b \in \bbE^{*}$ with $\rmN(a) = \rmN(b) = 1$ and $\rmT(a^{2})
  = \rmT(b^{2}) = 0$. Then we have the following:
\begin{enumerate}[label=\rm{(\roman*)}, ref=\rm{(\roman*)}]
  \item \label{13} $\kappa_{1}(a^{2}, b^{2}) = \kappa_{3}(a^{2}, b^{2})$; 
  \item \label{02} $\kappa_{0}(a^{2}, b^{2}) - \kappa_{2}(a^{2},
    b^{2})= q \cdot \chi_{2}(2\rmT(ab));$ 
  \item \label{trans} for each $c \in \bbE^{*} \setminus \{a,b\}$ with
    $\rmN(c) = 1$, $\rmT(c^{2}) = 0$, we have
    \[
    \kappa_{0}(a^{2},c^{2})-\kappa_{2}(a^{2},c^{2}) = \kappa_{0}(b^{2},
    c^{2}) - \kappa_{2}(b^{2}, c^{2}) \iff \chi_{2}(2\rmT(ab)) \neq -1.
    \]
  \end{enumerate}
\end{corollary}
\begin{proof}
  We have \ref{13} following directly from Lemma~\ref{kappazzbar}.
 
  To see that $\ref{02}$ holds, notice that since $a$ and $b$ have
  norm $1$ they are both squares, $\chi_{2}(ab) = 1$; then by
  Theorem~\ref{kappa1minus} we have 
  \begin{align*}
  \kappa_{0}(a^{2}, b^{2}) - \kappa_{2}(a^{2}, b^{2}) & = 
    \kappa_{\chi_{2}(2)}(ab) - \kappa_{-\chi_{2}(2)}(ab) \\ & = 
    \chi_{2}(2)\left( \kappa_{1}(ab) - \kappa_{-1}(ab) \right)\\ & =
    q\cdot\chi_{2}(2\rmT(ab)).
  \end{align*}

  Finally to obtain \ref{trans}, consider $c \in \bbE^{*} \setminus
  \{a,b\}$ with $\rmN(c)  = 1$ and $\rmT(c^{2}) = 0$.  We can assume
  that $a \neq b$ since if so,
  $\kappa_{0}(a^{2},c^{2})-\kappa_{2}(a^{2},c^{2}) = \kappa_{0}(b^{2},
  c^{2}) - \kappa_{2}(b^{2}, c^{2})$ and $\chi_{2}(2\rmT(ab))=0$.
  Then by Lemma~\ref{SquareConic}, $\{a,b,c\}$ forms a basis for
  $\bbE$ over $\bbF$, and so applying Corollary~\ref{Discabc},
  $\chi_{2}(2\rmT(ab)\rmT(ac)\rmT(bc))=1$. By \ref{02},  
  \[
  \kappa_{0}(a^{2},c^{2})-\kappa_{2}(a^{2},c^{2}) = \kappa_{0}(b^{2},
  c^{2}) - \kappa_{2}(b^{2}, c^{2})
  \]
  if and only if $\chi_{2}(2\rmT(ac)) = \chi_{2}(2\rmT(bc))$ and, since
  both are nonzero, this implies that $\chi_{2}(\rmT(ac))
  \chi_{2}(\rmT(bc)) =1$.  Thus we also have $\chi_{2}(2\rmT(ab))=1$.
\end{proof}

This result shows that 
$\bm{B}_{1} = \bm{B}_{3}$ so we have that Equation \eqref{eqHdef}
becomes
\begin{equation}
  \bm{H} =  \bm{B}_{0} - \bm{B}_{2} - \bm{I} = 
  \begin{bmatrix} 
    \kappa_{0}(a_{i}^{2}, a_{j}^{2}) -
    \kappa_{2}(a_{i}^{2}, a_{j}^{2}) 
  \end{bmatrix} - \bm{I}.
\end{equation}
Applying Corollary~\ref{ThBig}~\ref{02}, we have
\begin{equation}\label{eqH}
  \bm{H} = 
  \begin{bmatrix} q \cdot \chi_{2}(2\rmT(a_{i}a_{j}))\end{bmatrix} -
  \bm{I}.
\end{equation}

\begin{lemma}\label{LemH}
Let $\bm{H}$ be as in Equation~\eqref{eqH}, with the rows and columns
indexed by the $a_{i}$ from Equation~\eqref{eqa} in the natural way.
Put 
\begin{align*}
  X_{1} & = \{ a_{i} : \ 1 \leq i \leq q+1 \ | \
  \chi_{2}(2\rmT(a_{1}a_{i})) \neq -1 \} \mbox{ and}\\  
  X_{2} & = \{  a_{i} : \ 1 \leq i \leq q+1 \ | \
  \chi_{2}(2\rmT(a_{1}a_{i}))=-1 \},
\end{align*}
and let $\bm{K}$ be the adjacency matrix of the graph $K_{X_{1}}
\oplus K_{X_{2}}$ and $\bm{K}^{\prime}$ be the adjacency matrix of 
the complementary graph.  Then
\begin{equation}\label{eqHfinal}
  \bm{H} = q \bm{K} - q \bm{K}^{\prime} - I.
\end{equation}
\end{lemma}
\begin{proof}
  It is clear that the diagonal entries of $\bm{H}$ are $-1$, so we
  only need to consider the entries $\bm{H}_{ij}$ for $i \neq j$;
  also, since $\bm{H}$ is symmetric, we can assume that $i < j$. All
  entries of $\bm{H}$ that are not on the diagonal are equal to $\pm
  q$.
  
  Assume that $\bm{H}_{ij} = -q$; then $\chi_{2}(2\rmT(a_{i}a_{j})) =
  -1$.  If $i=1$, then $a_{i} \in X_{1}$ and
  $\chi_{2}(2\rmT(a_{1}a_{j}))=-1$, so $a_{j} \in X_{2}$. On the other
  hand, if $i \neq 1$ then $a_{1} \in \bbE^{*} \setminus \{ a_{i},
  a_{j} \}$ and, applying Corollary~\ref{ThBig}~\ref{trans},
  we must have $\chi_{2}(2\rmT(a_{1}a_{i})) \neq
  \chi_{2}(2\rmT(a_{1}a_{j}))$, with both nonzero by
  Corollary~\ref{Discabc}.  Therefore we can assume (WLOG) that $a_{i}
  \in X_{1}$ and $a_{j} \in X_{2}$. 
\end{proof}

\subsection{Proof of the main theorem}

\begin{theorem}\label{th:final}
  Let $q = p^{h}$ be a prime power with $q \equiv 5 \mbox{ or } 9 
  \bmod{12}$.  Then the hyperbolic quadric $\cQ^{+}(5,q)$ has a
  decomposition $\pi_{1} \cup \pi_{2} \cup \cT_{1} \cup \cT_{2}$ into 
  disjoint tight sets, with $\pi_{1}$ and $\pi_{2}$ being generators, 
  and $\cT_{1} \simeq \cT_{2}$ being
  $\frac{1}{2}(q^{2} -1)$-tight sets. 
\end{theorem}
\begin{proof}
  Define the sets $X_{1}$ and $X_{2}$ and matrices $\bm{K}$ and 
  $\bm{K}^{\prime}$ and $\bm{H}$ as in Lemma~\ref{LemH}; then
  eigenvectors of the matrix $\bm{H} =  q \cdot \bm{K} - q \cdot
  \bm{K}^{\prime} - I$ can be used to obtain eigenvectors of the
  matrix $\bm{B}$ defined in Lemma~\ref{TDMat}, which is a column sum
  matrix associated to a tactical decomposition of the collinearity
  matrix $\bm{A}$ of  $\cQ^{+}(5,q)$ after throwing away rows and
  columns corresponding to points in disjoint generators 
  $\pi_{1}$ and $\pi_{2}$. 

  Now we form the vector 
  \[
  \bm{v}  = \bm{c}_{X_{1}} - \frac{1}{2} \bm{j}
  = \frac{1}{2} \left( \bm{c}_{X_{1}} -
    \bm{c}_{X_{2}}\right),
  \]
  where $\bm{c}_{X_{1}}$ and $\bm{c}_{X_{2}}$ are the characteristic
  vectors of the sets $X_{1}$ and $X_{2}$.
  We have that
  \begin{align*}
    \displaystyle \bm{v} \bm{H}
    & =  \frac{1}{2}\left( \bm{c}_{X_{1}} - \bm{c}_{X_{2}}\right)
    \left( q\bm{K} - q\bm{K}^{\prime} - I\right)\\ 
    & =  \frac{1}{2} \lbrack
    \left( q \left( |X_{1}| -1 \right) \bm{c}_{X_{1}}
      - q |X_{1}| \bm{c}_{X_{2}} - \bm{c}_{X_{1}} \right) 
    - \left( q \left( |X_{2}| -1 \right) \bm{c}_{X_{2}}
      - q |X_{2}| \bm{c}_{X_{1}} - \bm{c}_{X_{2}}
    \right)
    \rbrack\\
    & =  \frac{1}{2}\left( q \left( |X_{1}| + |X_{2}| -1 \right) -1
    \right) \bm{c}_{X_{1}} - \left( \left( |X_{1}| + |X_{2}| -1 \right)
      -1 \right) \bm{c}_{X_{2}} \\ 
    & =  \frac{1}{2}(q^2 -1) (\bm{c}_{X_{1}} - \bm{c}_{X_{2}})
    \mbox{.} 
  \end{align*}
  So $\bm{v}$ is an eigenvector of $\bm{H}$ for eigenvalue $(q^{2} -1)$ and
  by Theorem~\ref{BCirc}, putting 
  \begin{align*}
    \bm{w}_{1} & = \begin{bmatrix*}[r]
      \bm{v}\\0\\-\bm{v}\\0 \end{bmatrix*}, & 
    \bm{w}_{2} & =\begin{bmatrix*}[r]
      0\\ \bm{v}\\0\\-\bm{v} \end{bmatrix*}, 
  \end{align*}
  we have that $\bm{w}_{1}$ and $\bm{w}_{2}$ are eigenvectors of
  $\bm{B}$ for this same eigenvalue.  Applying Theorem~\ref{EvecMod},
  vectors of the form 
  \[
  \bm{w} = \pm \bm{w}_{1} \pm \bm{w}_{2}
  \]
  correspond to $\frac{1}{2}(q^{2} -1)$-tight sets of $\cQ^{+}(5,q)$
  which are disjoint from $(\pi_{1} \cup \pi_{2})$.  

  Put $\cT_{1}$ to be the tight set corresponding to
  $\bm{w}_{1}+\bm{w}_{2}$ and $\cT_{2}$ to be the tight set
  corresponding to $-\bm{w}_{1}-\bm{w}_{2}$; then it is clear that
  $\cT_{1}$ and $\cT_{2}$ are disjoint.  Comparing the vectors
  $\bm{w}_{1}$ and $\bm{w}_{2}$ to the ordering of the entries of
  $\bm{B}$ shown in Equation~\eqref{eqOrbitOrder}, we see that
  \begin{align*}
    \cT_{1} = & 
    \left( \cup_{a_{i} \in X_{1}}  \left( (1, a_{i})^{G} \cup (1,
        \omega a_{i})^{G} \right) \right) \cup 
    \left( \cup_{a_{j} \in X_{2}} \left( (1, \omega^{2} a_{j})^{G} \cup
        (1, \omega^{3} a_{j})^{G} \right) \right) \mbox{ and}\\
    \cT_{2} = & 
    \left( \cup_{a_{i} \in X_{2}}  \left( (1, a_{i})^{G} \cup (1,
        \omega a_{i})^{G} \right) \right) \cup 
    \left( \cup_{a_{j} \in X_{1}} \left( (1, \omega^{2} a_{j})^{G} \cup
        (1, \omega^{3} a_{j})^{G} \right) \right),
  \end{align*}
  where $G$ is the group defined in Lemma~\ref{GrpOrbs}.
  It is clear that the projective similarity $(x,y) \mapsto (x,
  \omega^{2} y)$ sends the set $\cT_{1}$ to $\cT_{2}$.
\end{proof}

Note that there are a couple of decisions made in the proof of this
theorem.  The first is the ordering of the $a_{i}$; the choice of
$a_{1}$ affects the definitions of the sets $X_{1}$ and $X_{2}$.
However, ordering these elements differently can only possibly
interchange the role of $X_{1}$ and $X_{2}$ (and so interchange the
tight sets $\cT_{1}$ and $\cT_{2}$).  Also, in defining the two tight
sets, we could just as easily put $\cT_{1}^{\prime}$ to be the tight
set corresponding to $\bm{w}_{1}-\bm{w}_{2}$ and $\cT_{2}^{\prime}$ 
to be the tight set corresponding to $-\bm{w}_{1}+\bm{w}_{2}$; this
would make
\begin{align*}
    \cT_{1}^{\prime} = & 
    \left( \cup_{a_{i} \in X_{1}}  \left( (1, a_{i})^{G} \cup (1,
        \omega^{3} a_{i})^{G} \right) \right) \cup 
    \left( \cup_{a_{j} \in X_{2}} \left( (1, \omega^{2} a_{j})^{j} \cup
        (1, \omega a_{j})^{G} \right) \right) \mbox{ and}\\
    \cT_{2}^{\prime} = & 
    \left( \cup_{a_{i} \in X_{2}}  \left( (1, a_{i})^{G} \cup (1,
        \omega^{3} a_{i})^{G} \right) \right) \cup 
    \left( \cup_{a_{j} \in X_{1}} \left( (1, \omega^{2} a_{j})^{j} \cup
        (1, \omega a_{j})^{G} \right) \right).
  \end{align*}
  It is clear in this case that, under the projective similarity
  $(x,y) \mapsto(x, \omega y)$, $\cT_{1}^{\prime} \mapsto \cT_{1}$ and 
  $\cT_{2}^{\prime} \mapsto \cT_{2}$.
 
Let $\phi: \bbF_{q^3} \rightarrow \bbF_{q^3}$: $x \mapsto x^q$. Hence $\phi$ has order $3$. 
The map $e: (u,v) \mapsto (\phi(u),\phi(v))$ is a semi-similarity of the formed space $(V,f)$. 
It induces a collineation of $\cQ^+(5,q)$. It is straightforward to check that $\phi$ is a permutation
of $\cS$. But since $T(\phi(x)) = T(x)$, and $\phi$ maps the squares of $\bbF_{q^3}$ to squares, 
$\phi$ is also a permutation of the sets $X_1$ and $X_2$.  
Also, it can be seen that the map $o: (u,v) \mapsto (v, \omega u)$ is
a collineation of $\cQ^{+}(5,q)$ that stabilizes $\cT_{1}$ and $\cT_{2}$.
The following corollary is now obvious.
   
\begin{corollary}
The tight sets constructed in Theorem~\ref{th:final} are stabilized by
the group $\langle c,z,e,o \rangle$. 
This group has order $3\frac{(q-1)}{2}(q^2+q+1)$. 
\end{corollary}  
Note that the map $o$ interchanges the two sets of generators of
$\cQ^{+}(5,q)$, and thus does not induce a collineation of $\pg(3,q)$
under the Klein correspondence.

\section{The Klein correspondence and Cameron--Liebler line classes}
The Klein correspondence maps points of $\cQ^{+}(5,q)$ to lines of
$\pg(3,q)$, with collinear points of $\cQ^{+}(5,q)$ corresponding to
intersecting lines of $\pg(3,q)$.  Each generator of $\cQ^{+}(5,q)$
then corresponds to a set of $q^{2}+q+1$ pairwise intersecting lines
of $\pg(3,q)$, and so it's image is either of the form $\st(\bm{p})$
for some point $\bm{p}$ or $\lne(\pi)$ for some plane $\pi$ depending
on the system the generator came from. We will assume that the
generators in the same system as $\pi_{1}$ correspond to the point
stars of $\pg(3,q)$, and that the generators in the same system as
$\pi_{2}$ correspond to the planes.  It is then clear that any
collineation of $\cQ^{+}(5,q)$ induces either a collineation or a
correlation of $\pg(3,q)$ depending on whether it stabilizes or
interchanges the two systems of generators.

Under the Klein correspondence, a $x$-tight set corresponds to a set
of lines in $\pg(3,q)$ called a Cameron--Liebler line class with
parameter $x$; while there are many equivalent definitions, we will
find the following characterizations first described in
\cite{Pentt1991} most useful.
\begin{theorem}\label{Tintsizes}
A set of lines $\cL$ is a Cameron--Liebler line class with parameter
$x$ if and only if the following equivalent conditions are met:
\begin{enumerate}[label=\rm{(\roman*)}, ref=\rm{(\roman*)}]
\item for every incident point-plane pair $(\bm{r}, \tau)$ we have
\[
|\st(\bm{r}) \cap \cL| + |\lne(\tau) \cap \cL| = x+(q+1)|\pen(\bm{r},
\tau) \cap \cL|;
\]
\item for every line $\ell \in \pg(3,q)$, the number of lines $m \in
  \cL$ distinct from $\ell$ and intersecting $\ell$ nontrivially is
  given by
\[
x(q+1)+(q^{2}-1)\bm{c}_{\cL}(\ell).
\]
\end{enumerate}
\end{theorem}

The result of Theorem~\ref{th:final} gives a decomposition of the points
of $\cQ^{+}(5,q)$ into $\pi_{1}$, $\pi_{2}$, $\cT_{1}$, $\cT_{2}$,
where $\cT_{1}$ and $\cT_{2}$ are isomorphic $\frac{(q-1)}{2}$-tight
sets.  We have that $\cT_{1}$ and $\cT_{2}$ are stabilized by an
abelian group $G = \langle c, z \rangle$, where $c$ and $z$ are the
maps defined (as in Lemma~\ref{GrpOrbs}) by 
\begin{align*}
c : (x, y) & \mapsto (\mu x, \mu^{-1} y) & z : (x,y) & \mapsto (x,
\omega^{4} y).
\end{align*} 
Under the Klein correspondence, $\pi_{1}$ corresponds to the set
$\st(\bm{p})$ of lines through a common point $\bm{p}$ in $\pg(3,q)$,
$\pi_{2}$ to the set $\lne(\pi)$ of lines in a common plane $\pi \not
\ni \bm{p}$, and $\cT_{1}$ and $\cT_{2}$ to isomorphic
Cameron--Liebler line classes $\cL_{1}$ and $\cL_{2}$ each having
parameter $\frac{(q^{2} -1)}{2}$.  The group $G = \langle c,z \rangle$
induces a collineation group of $\pg(3,q)$ (which for convenience we
will also call $G$) stabilizing these four sets of lines.

\begin{theorem}\label{thm:orbits}
The orbits of $G$ on the points of $\pg(3,q)$ are as follows: 
\begin{enumerate}[label=\rm{(\roman*)}, ref=\rm{(\roman*)}]
\item $\bm{p}$ is fixed;
\item the points on $\pi$ all fall into a single orbit of size
  $q^{2}+q+1$; 
\item the remaining points fall into four orbits, each having size
  $\frac{(q-1)}{4} (q^{2}+q+1)$. 
\end{enumerate}
\end{theorem}
\begin{proof}
We will prove this by considering the action of $G$ on the system of
generators of $\cQ^{+}(5,q)$ corresponding to points in $\pg(3,q)$.
Any generator in this system different from $\pi_{1}$ must meet
$\pi_{1}$ in a single point, and must either be disjoint from
$\pi_{2}$ or else meet $\pi_{2}$ in a line.  Since $S$ is transitive
on the points of $\pi_{1}$, we consider an arbitrary point $\bm{s} =
(\mu^{s}, 0) \in \pi_{1}$.  There are $q+1$ generators on $\bm{s}$ in
each system; the system we are interested in contains $\pi_{1}$, a
unique generator meeting $\pi_{2}$ in a line, and $q-1$ others.  The
stabilizer in $G$ of $\bm{s}$ is simply $\langle z \rangle$, which
fixes $\pi_{1}$ and stabilizes the unique plane on $\bm{s}$ which
meets $\pi_{2}$ in a line.  Consider one of the other planes
$\pi^{\prime} = \langle \bm{s}, (x_{1}, y_{1}), (x_{2}, y_{2})
\rangle$.  We must have 
\begin{align*}
\rmT(\mu^{s}y_{1}) = \rmT(\mu^{s}y_{2}) = \rmT(x_{1}y_{1}) =
\rmT(x_{2}y_{2}) = \rmT(x_{1}y_{2}) + \rmT(x_{2}y_{1}) = 0.
\end{align*}
Since $\pi^{\prime}$ is disjoint from $\pi_{2} = \{ (0, a) : a \in
\bbE^{*} \}$, we must have that $\{ \mu^{s}, x_{1}, x_{2} \}$ is
linearly independent over $\bbF$.  Therefore, 
\[
x_{1} \not \in \{ c : \rmT(c y_{2}) = 0 \} = \langle \mu^{s}, x_{2}
\rangle.
\]
Now $z^{k}$ maps $(x_{2}y_{2})$ to $(x_{2}, \omega^{4k} y_{2})$,
and since $\rmT(x_{1}y_{2}) \neq 0$,
\[
\rmT(\omega^{4k}x_{1}y_{2}) + \rmT(x_{2}y_{1}) = (\omega^{4k}
-1)\rmT(x_{1}y_{2}) \neq 0
\]
unless $\omega^{4k} = 1$, in which case $z^{k}$ is the identity map.
Thus $\langle z \rangle$ acts semiregularly on the planes through
$\bm{s}$ distinct from $\pi_{1}$ and disjoint from $\pi_{2}$.  The
result follows immediately from $| \langle z \rangle|$, the
transitivity of $G$ on $\pi_{1}$, and the Klein mapping from
$\cQ^{+}(5,q)$ to $\pg(3,q)$.
\end{proof}

We want to consider the number of lines of $\cL_{1}$ through the
various points of $\pg(3,q)$; some information can be deduced
immediately. 
\begin{lemma} \leavevmode
\begin{enumerate}[label=\rm{(\roman*)}, ref=\rm{(\roman*)}]
\item $|\lne(\pi) \cap \cL_{1}| = 0$; \label{LIpiint}
\item $|\st(\bm{p}) \cap \cL_{1} = 0$; \label{LIpint}
\item for any plane $\tau \ni \bm{p}$, $|\lne(\tau) \cap \cL_{1}|
  =\frac{(q^{2}-1)}{2}$; \label{LItauint}
\item for any point $\bm{r} \in \pi$, $|\st(\bm{r}) \cap \cL_{1}| =
  \frac{(q^{2} -1)}{2}$. \label{LIrint}
\end{enumerate}
\end{lemma}
\begin{proof}
\ref{LIpiint} and \ref{LIpint} follow directly from our construction
of $\cT_{1}$. \ref{LItauint} and \ref{LIrint} follow from applying
Theorem~\ref{Tintsizes} to $(\bm{p}, \tau)$ and $(\bm{r}, \pi)$,
respectively.
\end{proof}
We need to find the number of lines of $\cL_{1}$ through the points
$\bm{r} \neq \bm{p}$ with $\bm{r} \not\in \pi$.  If we apply
Theorem~\ref{Tintsizes} to $(\bm{r}, \tau)$ for any plane $\tau$ on
$\ell = \langle \bm{p}, \bm{r} \rangle$, we see that
\[
|\st(\bm{r}) \cap \cL_{1}| = (q+1)|\pen(\bm{r}, \tau) \cap \cL_{1}|.
\]
This tells us that $|\st(\bm{r}) \cap \cL_{1}|$ is divisible by
$(q+1)$.  Since the points $\bm{r}$ fall into four orbits under $G$,
we will let 
\[
(q+1)a_{1}\leq (q+1)a_{2} \leq (q+1)a_{3} \leq (q+1)a_{4}
\] 
be the values of $|\st(\bm{r}) \cap \cL|$ for $\bm{r}$ in each of
these four orbits. We can also see from this formula that $|
\pen(\bm{r}, \tau) \cap \cL_{1}|$ depends only on $\bm{r}$ and not on
the choice of the plane $\tau$ containing $\ell$.  We can gain some
information on the values of $a_{1}, \ldots, a_{4}$ by utilizing the 
following concept introduced by Gavrilyuk and Mogilnykh in
\cite{GMCL4}. 
\begin{definition}
Let $\cL$ be a Cameron--Liebler line class with parameter $x$ in
$\pg(3,q)$, and let $\ell$ be a line of $\pg(3,q)$.  Number the points
on $\ell$ as $\bm{p}_{1},\ldots \bm{p}_{q+1}$, and the planes
containing $\ell$ as $\pi_{1}, \ldots \pi_{q+1}$. We define the
\emph{pattern} of $\cL$ with respect to $\ell$ to be the
$(q+1)\times(q+1)$ matrix $\cT(\ell)$ given by
\[
\cT(\ell) = 
\begin{bmatrix}
t_{1,1} & \ldots & t_{1, q+1}\\
\vdots & \ldots & \vdots\\
t_{q+1,1} & \ldots & t_{q+1,q+1}
\end{bmatrix},
\]
where $t_{i,j} = | \cL \cap \pen(\bm{p_{i}}, \pi_{j}) \setminus \{
\ell \}|$.  
\end{definition}
\begin{theorem}\label{Tpatternprops}
For every line $\ell \in \pg(3,q)$, the pattern $\cT(\ell)$ has the
following properties:
\begin{enumerate}[label=\rm{(\roman*)}, ref=\rm{(\roman*)}]
\item $0 \leq t_{i,j} \leq q$ for all $i,j$;
\item $\displaystyle \sum_{i,j} t_{i,j} = x(q+1)
  +\bm{c}_{\cL}(\ell)(q^{2} -1)$; 
\item $\displaystyle \sum_{j} t_{k,j} + \sum_{i} t_{i,l} =
  x+(q+1)t_{k,l} + (q-1)\bm{c}_{\cL}(\ell)$ for all $k,l$;
\item \label{Isquaresums} $\displaystyle \sum_{i,j} t_{i,j}^{2} = (x -
  \bm{c}_{\cL}(\ell))^{2} + q(x - \bm{c}_{\cL}(\ell)) +
  \bm{c}_{\cL}(\ell)q^{2}(q+1)$.
\end{enumerate}
\end{theorem}

\begin{lemma}
Let $\ell$ be a line on $\bm{p}$; we will number the points on
$\ell$ as $\bm{r}_{1}, \ldots, \bm{r}_{q+1}$ and the planes as
$\tau_{1}, \ldots, \tau_{q+1}$.  We will take $\bm{r}_{1} = \bm{p}$,
$\bm{r}_{2} = \ell \cap \pi$, and $\tau_{1} = \pi$.  Then the pattern
of $\cL_{1}$ with respect to $\ell$ is given (up to a permutation of the
columns) by 
\[
\mathcal{T}(\ell) = 
\begin{bmatrix}
0 & \frac{q-1}{2} & a_{1} & \ldots & a_{1} & a_{2} & \ldots & a_{2} & 
a_{3} & \ldots & a_{3} & a_{4} & \ldots & a_{4}\\
\vdots &\vdots & \vdots & & \vdots & \vdots & & \vdots & \vdots & &
\vdots & \vdots & & \vdots \\
0 & \frac{q-1}{2} & a_{1} & \ldots & a_{1} & a_{2} & \ldots & a_{2} &
a_{3} & \ldots & a_{3} & a_{4} & \ldots & a_{4}
\end{bmatrix}^{T}
\]
where each $a_{i}$ is repeated in $\frac{q-1}{4}$ rows.
\end{lemma}
\begin{proof}
Since $|\st(\bm{p}) \cap \cL_{1}| = 0$, it is clear that
$|\pen(\bm{p}, \tau) \cap \cL_{1}| = 0$ for every plane $\tau$ on
$\ell$.  Also, $\bm{r}_{2} \in \pi$ so $|\st(\bm{r}_{2}) \cap
\cL_{1}$; applying Theorem~\ref{Tintsizes}, we see that $|
\pen(\bm{r}_{2}, \tau) \cap \cL_{1}| = \frac{q-1}{2}$ for all $\tau$
on $\ell$.  For $\bm{r}_{i}$, $2 < i \leq q+1$, the value $|
\pen(\bm{r}_{i}, \tau) \cap \cL_{1}| \in \{ a_{1}, a_{2}, a_{3}, a_{4}
\}$ depends only on which orbit $\bm{r}_{i}$ falls into under $G$. The
proof of Theorem~\ref{thm:orbits} makes it clear that the points of $\ell
\setminus \{\bm{r}_{1}, \bm{r}_{2}\}$ are evenly distributed among
these four orbits.
\end{proof}

\begin{lemma}\label{Lavals}
The values $a_{1}, a_{2}, a_{3}, a_{4}$ satisfy the following properties:
\begin{enumerate}[label=\rm{(\roman*)}, ref=\rm{(\roman*)}]
\item \label{LavalsIcomp} $a_{4} = q-a_{1}$ and $a_{3} = q-a_{2}$;
\item \label{LavalsIquad} $a_{1}(q-a_{1}) + a_{2}(q-a_{2}) =
  \frac{q(q-1)}{2}$; 
\item \label{LavalsIineq} $\frac{q - \sqrt{2q-1}}{2} \leq a_{1} \leq 
  \frac{(q - \sqrt{q})}{2} \leq a_{2} \leq   \frac{(q-1)}{2}$;
\end{enumerate}
\end{lemma}
\begin{proof}
\ref{LavalsIcomp} comes from noticing that, for any point $\bm{r} \neq
\bm{p}$ with $\bm{r} \not\in \pi$, every line through $\bm{r}$ that
does not contain $\bm{r}$ is in either $\cL_{1}$ or $\cL_{2}$;
therefore if $| \st(\bm{r}) \cap \cL_{1}| = a\leq \frac{q-1}{2}$, we
have $| \st(\bm{r}) \cap \cL_{2}| = (q-a) \geq \frac{q-1}{2}$.  Since
$\cL_{1} \simeq \cL_{2}$, this means there exists a point
$\bm{r}^{\prime}$ with $| \st(\bm{r}^{\prime}) \cap \cL_{1}| =
(q-a)$.  
We see that \ref{LavalsIquad} holds by applying
Theorem~\ref{Tpatternprops}~\ref{Isquaresums} and combining the result
with \ref{LavalsIcomp} using $x = \frac{(q^{2}-1)}{2}$.  
For \ref{LavalsIineq}, it is clear that $a_{2} \leq \frac{(q-1)}{2}$
since $a_{2} \leq a_{3} = (q-a_{2})$ (and $q$ is odd).  This implies
that $a_{2}(q-a_{2}) \leq \frac{(q^{2}-1)}{4}$ and so $a_{1}(q-a_{1})
\geq \frac{q^{2}-2q+1}{4}$, therefore $a_{1} \geq
\frac{q+\sqrt{2q-1}}{2}$.  Finally, if $a_{2} < \frac{(q -
  \sqrt{q})}{2}$ then $a_{2}(q-a_{2}) < \frac{q(q-1)}{4}$; this forces
us to have $a_{1}(q-a_{1})> \frac{q(q-1)}{4}$, which is impossible
since $a_{1} \leq a_{2}$.  Therefore we have $a_{1} \leq  \frac{(q -
  \sqrt{q})}{2} \leq a_{2}$.  
\end{proof}

A closer look at Lemma~\ref{Lavals}~\ref{LavalsIquad} tells us that 
\[
a_{2}^{2} \equiv -a_{1}^{2}  \bmod{q}.
\]
When $q \equiv 9 \bmod{12}$, we must have $q = 3^{2e}$ for some
integer $e$. In this case, the only way for $a_{2}^{2} \equiv
-a_{1}^{2}  \bmod{q}$ is for $a_{1} \equiv a_{2} \equiv 0
\bmod{\sqrt{q}}$; by Lemma~\ref{Lavals}~\ref{LavalsIineq}, this forces
$a_{1} = a_{2} = \frac{(q - \sqrt{q})}{2}$ (and so $a_{3} = a_{4} =
\frac{(q + \sqrt{q})}{2}$).

\begin{theorem}
For any positive integer $e$, there is a symmetric tactical
decomposition of $\pg(3,3^{2e})$ having four line classes given by
$\st(\bm{p})$, $\lne(\pi)$, $\cL_{1}$, $\cL_{2}$, and four point
classes given by $\{ \bm{p} \}$, $\pi$, 
$\cP_{1} = \{ \bm{r} \in \pg(3,q) |\ |\st(\bm{r}) \cap \cL_{1}| = (3^{2e}+1)(3^{2e}-3^{e}) \}$,
$\cP_{2} = \{ \bm{r} \in \pg(3,q) | \ |\st(\bm{r}) \cap\cL_{1}| = (3^{2e}+1) (3^{2e}+3^{e}) \}$.
\end{theorem}
\begin{proof}
The comments above make it clear that the given decomposition is
point-tactical, with Table~\ref{tab:STDLP} giving the number of lines 
from each line class on a point from each point class.  To see that
the decomposition is also block-tactical, consider a line $\ell$ in
$\pg(3,3^{2e})$.  If $\bm{p} \in \ell$, then $\ell$ contains $\bm{p}$
and a single point of $\pi$; by our remarks in the proof of 
Theorem~\ref{thm:orbits}, the remaining $(3^{2e}-1)$ points are
distributed evenly between $\cP_{1}$ and $\cP_{2}$.  If $\ell \subset
\pi$, then all $(3^{2e}+1)$ points of $\ell$ are in $\pi$. If $\ell
\in \cL_{1}$, then we have $\ell$ meeting $\pi$ in a single point and
the remaining $3^{2e}$ points on $\ell$ are either in $\cP_{1}$ or
$\cP_{2}$.  Let $m = |\ell \cap \cP_{1}|$ and $n = | \ell \cap
\cP_{2}|$, so $m+n = 3^{2e}$. Since $\ell \not\in \cL_{2}$, we have
that $\ell$ meets $\frac{(3^{2e}+1)(3^{4e}-1)}{2}$ lines of $\cL_{2}$;
the point in $\ell \cap \pi$ is on $\frac{(3^{4e}-1)}{2}$ lines of
$\cL_{2}$, the $m$ points in $\ell \cap \cP_{1}$ are each on
$\frac{(3^{2e}+1)(3^{2e}+3^{e})}{2}$, and the $n$ points in $\ell \cap
\cP_{2}$ are each on $\frac{(3^{2e}+1)(3^{2e}-3^{e})}{2}$. This allows
us to compute 
\[
\frac{(3^{4e}-1)}{2} + m \frac{(3^{2e}+1)(3^{2e}+3^{e})}{2} + n
\frac{(3^{2e}+1)(3^{2e}-3^{e})}{2} = \frac{(3^{2e}+1)(3^{4e}-1)}{2},
\mbox{ or}
\]
\[
m(3^{2e}+3^{e}) + (3^{2e}-m)(3^{2e}-3^{e}) = 3^{4e} -3^{2e}.
\]
From this, we see that $m = \frac{3^{2e}-3^{e}}{2}$ and $n =
\frac{3^{2e}+3^{e}}{2}$.  A similar argument allows us to compute the
intersection numbers of a line $\ell \in \cL_{2}$ with the four point 
classes, showing that our decomposition is block-tactical (with
Table~\ref{tab:STDPL} giving the number of points from each point
class on a line from each line class) as well as point-tactical. 
\end{proof}

\begin{table}[h]
\begin{center}
\begin{tabular}{|c|cccc|}
\hline
& $\st(\bm{p})$ & $\lne(\pi)$ & $\cL_{1}$ & $\cL_{2}$\\ \hline 
$\{ \bm{p} \}$ & $3^{4e}+3^{2e}+1$ & $0$ & $0$ & $0$ \\  
$\pi$ & $1$ & $3^{2e}+1$ & $\frac{3^{4e}-1}{2}$ & $\frac{3^{4e}-1}{2}$\\ 
$\cP_{1}$ & $1$   & $0$   & $\frac{(3^{2e}+1)(3^{2e}-3^{e})}{2}$ & $\frac{(3^{2e}+1)(3^{2e}+3^{e})}{2}$\\
$\cP_{2}$ & $1$   & $0$   & $\frac{(3^{2e}+1)(3^{2e}+3^{e})}{2}$ & $\frac{(3^{2e}+1)(3^{2e}-3^{e})}{2}$\\
\hline 
\end{tabular}
\end{center}
\caption{\label{tab:STDLP}{\em Lines per point for the symmetric
    tactical decomposition induced on $\pg(3,3^{2e})$.}}
\end{table}
\begin{table}[h]
\begin{center}
\begin{tabular}{|c|cccc|}
\hline
& $\st(\bm{p})$   & $\lne(\pi)$       & $\cL_{1}$ & $\cL_{2}$\\ \hline
$\{ \bm{p} \}$ & $1$  & $0$   & $0$    & $0$ \\
$\pi$ & $1$  & $3^{2e}+1$ & $1$   & $1$\\
$\cP_{1}$ & $\frac{3^{2e}-1}{2}$  & $0$   & $\frac{3^{2e}-3^{e}}{2}$   & $\frac{3^{2e}+3^{e}}{2}$\\
$\cP_{2}$ & $\frac{3^{2e}-1}{2}$  & $0$   & $\frac{3^{2e}+3^{e}}{2}$   & $\frac{3^{2e}-3^{e}}{2}$\\
\hline
\end{tabular}
\end{center}
\caption{\label{tab:STDPL}{\em Points per line for the symmetric
    tactical decomposition induced on $\pg(3,3^{2e})$.}}
\end{table}

A set of type $(m,n)$ in a projective or affine space is a set
$\mathcal{K}$ of points such that every line of the space contains
either $m$ or $n$ points of $\mathcal{K}$;  we require that $m < n$, 
and that both values occur (these sets are also frequently called
two-intersection sets or two-character sets).  For projective spaces,
there are many examples of these types of sets with $q$ both even and
odd.  However the situation is quite different for affine spaces.
When $q$ is even, we obtain a set of type $(0,2)$ in $\ag(2,q)$ from a
hyperoval of the corresponding projective plane, and similarly a set
of type $(0,n)$ from a maximal arc of degree $n$.  Examples of sets of
type $(m,n)$ in affine planes of odd order, on the other hand, are
extremely scarce. The only previously known examples are those
described in \cite{PR1995} (in affine planes of order $9$), along with
an example in $\ag(2,81)$ described in \cite{thesis} and
\cite{Rodgers}.  A result from \cite{PR1995} gives us the following:
\begin{lemma}\label{lem:affmnk}
If we have a set $\cK$ of type $(m,n)$ in an affine plane of order
$q$, then $k = |\cK|$ must satisfy  
\begin{align*}
k^2 -k(q(n+m-1) + n + m) + mnq(q +1) & = 0. 
\end{align*}
\end{lemma}

\begin{corollary}
There exists a set of type $(3^{2e}-3^{e}, 3^{2e}+3^{e})$ having size
$\frac{(3^{4e} - 3^{2e})}{2}$ in $\ag(2,3^{2e})$.
\end{corollary}
\begin{proof}
If we take a plane $\tau \neq \pi$ that does not contain $\bm{r}$
then all points of the affine plane $\tau^{\prime} = \tau \setminus (\tau \cap \pi)$
lie in either $\cP_{1}$ or $\cP_{2}$, and all lines lie in either
$\cL_{1}$ or $\cL_{2}$.  Putting $\cK = \cP_{1} \cap \tau^{\prime}$,
the values in Table~\ref{tab:STDPL} give the intersection numbers of
the lines of the plane with the set $\cK$. The size of $\cK$ follows
from Lemma~\ref{lem:affmnk}.
\end{proof}

\section{Final remarks}

While finishing this manuscript, Koji Momihara, Tao Feng and Qing Xiang informed us that they had proven at almost the same
time, and independently, the existence of a Cameron-Liebler line class of $\mathrm{PG}(3,q)$ with parameter $\frac{q^2-1}{2}$
for $q \equiv 5 \mbox{ or } 9 \bmod{12}$, see \cite{MFX}. They became aware of our result through the availability of the abstracts of the conference
``Combinatorics 2014'', held in june in Gaeta, Italy, where our result was presented, and were so kind to inform us about their result. 
As their approach is slightly more algebraic, and our approach more geometric, we all decided that publishing both manuscripts 
independently from each other was justified.  
  

\end{document}